\documentclass[12pt]{article}
\usepackage{graphicx}
\usepackage{algorithm}
\usepackage{algpseudocode}[compatible]

\usepackage{xcolor}    
\usepackage{comment}
\usepackage{indentfirst,csquotes}
\usepackage{ulem}
\usepackage{amssymb,amsthm,amsmath}
\usepackage{xcolor,paralist,hyperref,titlesec,fancyhdr,etoolbox}
\usepackage{mathtools}

\newtheorem{theorem}{Theorem}[]

\newtheorem{proposition}{Proposition}

\begin{document}

\title{A Generative Model-Free Form Deformation Approach for the Generation of Mesh Motions with Applications to PDE
}
\author{
Guglielmo Padula \\
SISSA \\
\texttt{gpadula@sissa.it}
\and
Artem Sinitsa \\
SISSA \\
\texttt{asinitsa@sissa.it}
\and
Gianluigi Rozza \\
SISSA \\
\texttt{grozza@sissa.it}
}
\maketitle


\begin{abstract}
We introduce a topology-agnostic framework for matching deformations of three-dimensional shapes with non-isomorphic mesh graphs by modelling the deformation as the flow of an Ordinary Differential Equation (ODE). The velocity field is parameterised by a time-dependent Free Form Deformation (FFD), expressed through displacements of a coarse control lattice, yielding a smooth and low-dimensional representation that decouples the deformation model from the discretisation of the source and target surfaces. Under mild regularity assumptions, we prove that the induced ODE map is a universal approximator (in the sup norm) for mappings between genus-0 surfaces, providing a theoretical expressivity guarantee. To further compress the representation and enable probabilistic inference, we couple the ODE--FFD model with a flow-based generative approach in the TarFlow framework, learning a compact latent parametrisation over time series of FFD maps. The resulting method supports efficient sampling and optimisation of plausible deformation trajectories while preserving mesh quality, and it enables scalable reduced-order modelling. Experiments on deforming-body flow benchmarks demonstrate improved accuracy and computational efficiency of reduced-order models constructed from the learned latent dynamics.
\end{abstract}

\section{Introduction}
\label{intro}

Matching and deforming three-dimensional shapes is central to many engineering and scientific problems, including shape optimisation \cite{demo_hull_2021}, registration of temporal scans \cite{tenderini_deformable_2025}, and reduced order modelling for fluid–structure interaction \cite{nkana_ngan_hybrid_2025}. Classical geometric morphing techniques such as Free Form Deformation (FFD) provide an intuitive control-based parametrisation of deformation fields \cite{sederberg_free-form_1986}, while modern model reduction and generative methods offer powerful tools for parameter compression \cite{padula_generative_2024,padula_generative_2025}. In this work, we bridge these two directions by modelling a time-dependent FFD map as the drift of an ODE to represent deformation trajectories and by introducing a data-driven generative model to reduce the parametrisation of such trajectories.

\subsection{Motivation and contributions}
Representations for deforming surfaces must satisfy two competing desiderata: expressivity (ability to represent complex deformations) and parsimony (low-dimensional control for tractable inversion, sampling, and reduced order models construction) \cite{padula_generative_2024}. FFDs have long been used for geometric control because of their smoothness and intuitive control lattice \cite{sederberg_free-form_1986,salmoiraghi_free-form_2018}. However, naively selecting a high-dimensional control grid leads to very large parameter spaces. Our central idea is to let the control lattice displacements vary in time and to define the deformation trajectory implicitly as the flow of an ODE whose drift field is exactly the instantaneous FFD mapping. Concretely, if \(x(t)\) denotes a material point of the object subject to deformation, we model its evolution by
\[
\frac{d}{dt} x(t) \;=\; \mathcal{F}\big(x(t); \mathbf{c}(t)\big),
\]
where \(\mathcal{F}\) represents the velocity field (drift term) induced by the FFD, parametrised by the time-dependent control vector \(\mathbf{c}(t)\). This formulation treats deformation as a continuous time process and naturally accommodates time integration, variational formulations, and coupling with PDE solvers in Eulerian or Lagrangian descriptions \cite{richter_fluid-structure_2017,glowinski_exact_2008}, furthermore allowing the use of a coarser lattice.

Our main contributions are:
\begin{enumerate}
  \item A novel ODE-with-FFD-drift framework, called FFD-ODE, to represent time-dependent shape deformations that supports base and target meshes of genus-0 meshes with different vertex topology, while preserving smoothness and invertibility properties in practice. We develop a theoretical result showing that, under mild regularity of control basis functions and sufficient control resolution, the FFD-ODE can approximate arbitrary smooth deformation flows on genus-$0$ surfaces, extending classical universal approximation ideas to the manifold/mesh setting \cite{cybenko_approximation_1989,hornik_approximation_1991}.
  \item A generative model based on the novel architecture TARFlow \cite{gu_starflow_2025,zhai_normalizing_2025} that maps a low-dimensional latent time series to sequences of control displacements \(\mathbf{c}(t)\) for the FFD-ODE system, enabling sampling, interpolation, and statistical validation of quantities of interest (QoIs) across deformation ensembles.
  
  \end{enumerate}
We perform numerical validation on two physics-based benchmark problems. The first one involves the the study of an incompressible-flow configuration involving flow past a deforming Stanford Bunny, solved via a stabilized Arbitrary Lagrangian–Eulerian (ALE) variational formulation with Streamline-Upwind Petrov-Galerkin (SUPG) stabilization implemented in FEniCSx \cite{brooks_streamline_1982,baratta_dolfinx_2023,scroggs_basix_2022,alnaes_unified_2014,turk_zippered_1994}. The second one is focused on a compressible hypersonic flow configuration involving flow past a deforming sphere, simulated using the OpenFOAM software package \cite{weller_tensorial_1998,anderson_hypersonic_2019}. We show that Reduced Order Models (ROMs) constructed using the proposed generative, reduced parametrisation exhibit enhanced offline/online computational efficiency and improved accuracy in QoIs compared with ROMs derived from naive, high-dimensional parametrisations \cite{quarteroni_reduced_2016,willcox_balanced_2002}.

In the context of inverse problems, such as aerodynamic shape optimisation or flow control,
standard deformation techniques often lead to invalid meshes or require expensive re-meshing \cite{mm_mesh_2016}. Our ODE-based framework naturally ensures topology preservation, while the generative model acts as a robust prior, effectively regularising the ill-posed inverse problem by restricting the search space to a learned manifold of plausible motions.

\subsection{Related work}
Our work connects classical geometric morphing techniques with modern generative and reduced-order modelling frameworks for PDE-driven simulations.
Generative models \cite{kingma_auto-encoding_2013,rezende_variational_2015,goodfellow_generative_2020,ho_denoising_2020} have become powerful surrogates for reduced-order modelling and data-driven flow synthesis, with diffusion and flow-based models showing strong expressivity for geometric and physical field generation. Wang et al.~\cite{potamias_shapefusion_2025} proposed a localised 3D diffusion model based on PCA and the octree-based diffusion method of~\cite{xiong_octfusion_2025}, which enables diffusion sharing across octree levels. TARFlow~\cite{zhai_normalizing_2025} revisited normalising flows, demonstrating their renewed potential as high-capacity generative density estimators. Finally, the recent surveys by Wang et al.~\cite{wang_recent_2024,wang_diffusion_2025} provide comprehensive overviews of the intersection between machine learning, computational fluid dynamics, and generative 3D modelling by emphasising the rise of hybrid, physics-aware generative frameworks. These advances collectively motivate our formulation, which combines the interpretability of FFD-based geometric control with the flexibility of generative latent representations for time-dependent mesh deformations and reduced order modelling.

Regarding geometrical parametrisation, FFD remains a cornerstone in geometric control and shape morphing, providing smooth and parameter-efficient deformation representations. Recent developments have extended the method toward higher fidelity and domain-specific applications. For instance, Chapelier et al.~\cite{chapelier_free-form_2021} introduced FFD-DIC, a Spline Regularised Free Form deformation framework for arbitrary finite element measurements, while Fukuda et al.~\cite{fukuda_efficient_2024} leveraged FFD-based deformation fields to accelerate musculoskeletal annotation in medical imaging. In the context of hydrodynamics, Liu et al.~\cite{liu_novel_2025} proposed a Body Fitted Free Form Deformation (b-FFD) method enabling smooth and large hull modifications by replacing the rectangular control lattice with a convex hexahedral control lattice. Generative models have also been applied for learning deformation of meshes of different topological connectivity and with constraints \cite{padula_generative_2024,padula_generative_2025}.
Amiri-Hezaveh et al.~\cite{amiri-hezaveh_physics-informed_2025} developed a physics-informed neural ODE approach for deformable medical image registration that jointly enforces data and physical constraints. These efforts highlight the growing role of ODE-based neural deformation frameworks for both interpretability and generative capacity.

Beyond explicit geometric control, recent studies have coupled deformation modelling with neural differential equations and generative priors. Tenderini et al.~\cite{tenderini_deformable_2025} presented an Auto-Decoder Neural ODE formulation for vascular anatomy registration and generative modelling. We adopt an approach similar to theirs, in the sense that we also leverage the Chamfer Distance for learning our map, and we also employ an ODE-based model. The main differences are the differences are that we use a FFD map instead of a neural network, by first learning a different map for each step and then constructing a generative model on the Free Form Deformation control displacements. The advantage of this approach is that it is more flexible: if one is interested in finding trajectories only to a few target shapes, there is no need to train a neural network (and thus to collect a lot of non-useful target shapes) as one can simply learn a different map for every shape in a short time.\\
Finally, Kabalan et al. \cite{kabalan_elasticity-based_2025} developed a technique to compute a shape-to-shape bijection employing an elasticity method, which, however, requires meshes with the same topology, while our methodology also works with point clouds with different number of points and without a topology. Furthermore, our method is guaranteed to produce a bijective morphism, as it corresponds to the solution of an ODE, while their method enforces the bijection constraint weakly with a penalisation term.

\subsection{Paper organization}
Section \ref{sec:method} introduces the ODE–FFD model, its mathematical properties, and the universality statement. In section \ref{sec:genmodel} we describe the generative model architectures and training losses used for deformation time series. Section \ref{sec:numerics} then details the two CFD benchmarks, solver setups (FEniCSx and OpenFOAM), and the ROM experiments. In Section \ref{sec:results}, we present numerical validations on QoI distributions and results about the performance of the ROMs. Finally, Section \ref{sec:conclusion} concludes and highlights future directions.

\section{The FFD-ODE model}
\label{sec:method}
In this section, we present the Free Form Deformation Ordinary Differential Equation (FFD-ODE) model, which is employed to learn a morphism between two genus-0 meshes that may differ in graph topology. This framework can be used either to obtain the morphism itself or, alternatively, to transfer the connectivity structure (topology) from the input mesh to the target mesh.

We begin by introducing the univariate Bernstein polynomials of degree \(m\) on the interval \([0,1]\), namely
\[
    b_{m, i}(x) = \binom{m}{i} x^i (1 - x)^{m - i}, \quad i = 0, \dots, m,
\]
and then define the associated trivariate Bernstein polynomials on \([0,1]^3\) by
\[
    B_{(i,j,k),(m,n,o)}(x,y,z) = b_{m,i}(x)\, b_{n,j}(y)\, b_{o,k}(z).
\]

A FFD map is a function from \([0,1]^3\) to \(\mathbb{R}^3\) of the form
\[
    \boldsymbol{F}_{\boldsymbol{\delta}}(\boldsymbol{x})
    = \boldsymbol{\delta}\, \boldsymbol{B}(\boldsymbol{x})
    := \sum_{i=0}^{m} \sum_{j=0}^{n} \sum_{k=0}^{o}
       \delta_{ijk}\, B_{(i,j,k),(m,n,o)}(x,y,z),
\]
where
   $ \delta_{ijk} \in \mathbb{R}^3$
are the FFD control parameters. Consequently, an FFD model possesses
\(3 \cdot (m+1) \cdot (n+1) \cdot (o+1)\) degrees of freedom.
In the remainder of the paper, we will treat the indices \(i,j,k\) as a single composite index, ordered in row-major fashion. Under this convention,
\(\boldsymbol{\delta} \in \mathbb{R}^{3 \times (m+1) \times (n+1) \times (o+1)}\) and
\(\boldsymbol{B}(\boldsymbol{x}) \in \mathbb{R}^{(m+1) \times (n+1) \times (o+1)}\) for all
\(\boldsymbol{x} \in [0,1]^3\).
For notational convenience, we set
\[
    \mathcal{C} := \mathbb{R}^{3 \times (m+1) \times (n+1) \times (o+1)}.
\]
If we choose the control points as
\[
    c_{ijk} = \left(\frac{i}{m}, \frac{j}{n}, \frac{k}{o}\right),
\]
then one can show that
\[
    \boldsymbol{F}_{\boldsymbol{c}}(\boldsymbol{x}) = \boldsymbol{x}
    \quad \forall\, \boldsymbol{x} \in [0,1]^3,
\]
i.e., this particular choice of control lattice yields the identity mapping on the unit cube.
 
By employing Bernstein polynomials, FFD attains a universal approximation property, as formalised by the following theorem \cite{foupouagnigni_multivariate_2020}.
\begin{theorem}
\label{thm:1}
For every $\epsilon>0$ and for every $\mathbf{f}\in C^{0}([0,1]^{3},\mathbb{R}^{3})$, there exist $m,n,o \in \mathbb{N}$ such that, setting
\[
    \delta_{ijk}=\mathbf{f}\left(\frac{i}{m}, \frac{j}{n}, \frac{k}{o}\right),
\]
the following inequality holds:
\[
\sup_{[0,1]^{3}}\left\lVert\mathbf{f}-\boldsymbol{F}_{\boldsymbol{\delta}}\right\rVert \le \epsilon. 
\]
\end{theorem}
Consider two closed, genus-0, boundary-less surfaces $S_{1}$ and $S_{2}$, both embedded in $\mathbb{R}^{3}$ and, in particular, contained in the unit cube, i.e., $S_{1}\subset[0,1]^{3}$ and $S_{2}\subset[0,1]^{3}$. Our objective is to construct a smooth deformation that maps $S_{1}$ onto $S_{2}$. The existence of such a deformation is guaranteed by classical results in differential topology \cite{chazal_condition_2005}. In particular we are interested in approximating the deformation from \(S_{1}\) to \(S_{2}\) by means of a dynamical system of the form
\[
\dot{\boldsymbol{x}}(t) = \boldsymbol{F}_{\boldsymbol{\alpha}(t)}(\boldsymbol{x}(t)).
\]
We denote by \(\boldsymbol{G}_{\boldsymbol{\alpha}}(\boldsymbol{x})\) the value at time \(t = 1\) of the solution to the above ODE with initial condition \(\boldsymbol{x}\).

We next recall the definition of the Chamfer Distance, which we employ as an error metric for measuring the quality of the approximation:
\[
\begin{gathered}
CD(S_{1},S_{2}) = \frac{1}{2}\frac{\displaystyle\int_{S_{1}}\left(\min _{y \in S_{2}}\|x-y\|\right)^2 \, dS(x)}{\displaystyle\int_{S_{1}} 1 \, dS(x)} \\
+ \frac{1}{2}\frac{\displaystyle\int_{S_{2}}\left(\min _{y \in S_{1}}\|x-y\|\right)^2 \, dS(x)}{\displaystyle\int_{S_{2}} 1 \, dS(x)}.
\end{gathered}
\]

The following theorem establishes that \(\boldsymbol{G}_{\boldsymbol{\alpha}}(\boldsymbol{x})\) satisfies a universal approximation property with respect to the Chamfer Distance.

\begin{theorem}
Let \(S_{1}\) and \(S_{2}\) be closed, smooth, genus-0 surfaces without boundary, both contained in \((0,1)^{3}\). Then, for every \(\epsilon > 0\) there exist \(m,n,o \in \mathbb{N}\) and functions \(\alpha_{ijk} : [0,1] \rightarrow \mathbb{R}^{3}\), with \(i = 0,\ldots,m\), \(j = 0,\ldots,n\), \(k = 0,\ldots,o\), such that 
\[
CD(\boldsymbol{G}_{\boldsymbol{\alpha}}(S_{1}), S_{2}) \le \epsilon,
\]
and
\[
\alpha_{ijk}(t) = 0 \quad \forall t \in [0,1] \text{  
}\mathrm{ whenever } \text{  
}i(m-i)j(n-j)k(o-k) = 0.
\]
\end{theorem}

Note that $\boldsymbol{G}_{\boldsymbol{\alpha}}$ depends highly nonlinearly on $\boldsymbol{\alpha}$. Consequently, it is advantageous to incorporate regularisation when estimating the optimal parameter vector $\boldsymbol{\alpha}$, in order not to get weights of very high norm.  

We now describe the procedure used to determine $\boldsymbol{\alpha}$. To obtain an initial guess $\boldsymbol{\alpha}_{0}(t)$ for the ensuing minimisation problem, we first compute a temporary variable $\boldsymbol{\beta}_{0}$ by minimising $\tilde{CD}(\boldsymbol{F}_{\boldsymbol{\beta}_{0}}(\tilde{S}_{1}), \tilde{S}_{2})$, where $\tilde{S}_{1}$ and $\tilde{S}_{2}$ are discretised approximations of $S_{1}$ and $S_{2}$, respectively, and
\[
\tilde{CD}(\tilde{S}_{1},\tilde{S}_{2})=\frac{1}{N}\sum_{i=1}^N \min _{1 \leq j \leq M}\left\|\boldsymbol{a}^{(i)}-\boldsymbol{b}^{(j)}\right\|^2+\frac{1}{M}\sum_{j=1}^M \min _{1 \leq i \leq N}\left\|\boldsymbol{a}^{(i)}-\boldsymbol{b}^{(j)}\right\|^2,
\]
with $\tilde{S}_{1}=\{\boldsymbol{a}_i \mid i=1,\ldots,N\}$ and $\tilde{S}_{2}=\{\boldsymbol{b}_j \mid j=1,\ldots,M\}$. The functional $\tilde{CD}$ thus represents a discrete approximation of the Chamfer distance.  

The resulting optimisation problem is solved using the L‑BFGS algorithm, with the initial iterate chosen as the control‑point configuration $\boldsymbol{c}$ such that $\boldsymbol{F}_{\boldsymbol{c}}(\boldsymbol{x})=\boldsymbol{x}$, i.e. the identity mapping.

The L‑BFGS method requires the gradient of $\tilde{CD}$ with respect to $\tilde{S}_{1}$, which is defined pointwise as
\[
 \frac{\partial\tilde{CD}}{\partial \boldsymbol{a}^{(i)}}(\tilde{S}_{1},\tilde{S}_{2})
 =\frac{2}{N}\left(\boldsymbol{a}^{(i)}-\boldsymbol{b}^{j_i}\right)
 +\frac{2}{M} \sum_{\left\{j: i_j=i\right\}}\left(\boldsymbol{a}^{(i)}-\boldsymbol{b}^{(j)}\right),
\]
where
\[
   j_i := \arg \min _j \left\|\boldsymbol{a}^{(i)}-\boldsymbol{b}^{(j)}\right\|^2
\]
and
\[
   i_{j} := \arg \min _i \left\|\boldsymbol{b}^{(j)}-\boldsymbol{a}^{(i)}\right\|^2.
\]

The $\operatorname{argmin}$ operators are evaluated by constructing an appropriate k‑d tree for efficient nearest‑neighbour queries.

Once $\boldsymbol{\beta}_{0}$ has been computed, we initialise
\[
    \boldsymbol{\alpha}^{(0)}(t) = \boldsymbol{\beta}_{0} - \boldsymbol{c}
    \quad \forall t \in [0,1].
\]

We then define
\[
\Phi(\tilde{S}) = \tilde{CD}(\tilde{S},\tilde{S}_{2})
\]
as the discretised Chamfer distance functional with fixed target set $\tilde{S}_{2}$.
The optimisation problem under consideration is
\[
\min\limits_{\boldsymbol{\alpha} \in C([0,1])} \, \Phi\bigl(\boldsymbol{G}_{\boldsymbol{\alpha}}(\tilde{S}_{1})\bigr)
\;+\; \rho \int_{0}^{1} \bigl(\boldsymbol{\alpha}(t) - \boldsymbol{\beta}_{0} + \boldsymbol{c}\bigr)^{2} \, dt,
\]
where $\rho > 0$ denotes a regularisation parameter.  

According to Pontryagin's Maximum Principle \cite{geering_optimal_2007}, a first-order necessary condition for optimality of this control problem is given by the following system:
\[
\begin{cases}
\dot{\boldsymbol{\lambda}}^{(i)}(t)
= -\dfrac{\partial H}{\partial \boldsymbol{y}^{(i)}}(t)
= -\bigl(\boldsymbol{\lambda}^{(i)}(t)\bigr)^{\top}\boldsymbol{\alpha}(t)\bigl(\nabla \boldsymbol{B}(\boldsymbol{y}^{(i)}(t))\bigr)^{\top},\\[0.6em]
\dot{\boldsymbol{y}}^{(i)}(t)
= \dfrac{\partial H}{\partial \boldsymbol{\lambda}^{(i)}}(t)
= \boldsymbol{\alpha}(t)\,\boldsymbol{B}\bigl(\boldsymbol{y}^{(i)}(t)\bigr),\\[0.4em]
\boldsymbol{y}^{(i)}(0) = \boldsymbol{a}^{(i)},\\[0.4em]
\boldsymbol{\lambda}^{(i)}(1)
= -\dfrac{\partial \Phi}{\partial \boldsymbol{a}^{(i)}}\bigl(\boldsymbol{G}_{\boldsymbol{\alpha}}(\tilde{S}_{1})\bigr),\\[0.6em]
\dfrac{\partial H}{\partial \boldsymbol{\alpha}}(t)
= 2\rho\bigl(\boldsymbol{\alpha}(t) - \boldsymbol{\beta}_{0} + \boldsymbol{c}\bigr)
+ \sum_{i=1}^N \boldsymbol{\lambda}^{(i)}(t)\,\boldsymbol{B}\bigl(\boldsymbol{y}^{(i)}(t)\bigr)^{\top}
= 0,
\end{cases}
\]
where $H$ denotes the Hamiltonian, given by
\[
H(Y, \boldsymbol{\alpha}, \Lambda)
= \rho \,\|\boldsymbol{\beta}_{0} - \boldsymbol{\alpha}\|^{2}
+ \sum_{i=1}^N \bigl(\boldsymbol{\lambda}^{(i)}\bigr)^{\top} \boldsymbol{\alpha}\,\boldsymbol{B}\bigl(\boldsymbol{y}^{(i)}\bigr),
\]
and
\[
\Lambda = \bigl\{\boldsymbol{\lambda}^{(i)} \,\big|\, i = 1,\ldots, N \bigr\}
\]
denotes the collection of costate variables.  We solve this system of ODEs by employing a forward–backward sweep method. Starting from an initial guess $\boldsymbol{\alpha}=\boldsymbol{\alpha}^{(0)}$, we first integrate $\boldsymbol{y}^{(i)}$ forward in time, then integrate $\boldsymbol{\lambda}^{(i)}$ backward in time, and subsequently update $\boldsymbol{\alpha}$ according to
\[
\boldsymbol{\alpha}(t)
=
\boldsymbol{\beta}_{0}
+
\frac{1}{2\rho}
\sum_{i=1}^N
\boldsymbol{B}\!\left(\boldsymbol{y}^{(i)}(t)\right)^{\!T}
\boldsymbol{\lambda}^{(i)}(t).
\]
This iterative procedure is repeated until convergence.

This algorithm constructs the trajectories of an arbitrary lattice defined on an arbitrary genus-0 mesh, which is morphed, via a velocity field induced by FFD, into another arbitrary genus-0 mesh.

The main advantage of this approach, compared with, for instance, a simple linear interpolation in time of the FFD control points, is the guaranteed injectivity of the resulting mapping. This property is particularly relevant for the stable and accurate numerical solution of partial differential equations.

As an illustrative application, see Figs.~\ref{fig:1}, \ref{fig:bunny_traje}, and \ref{fig:cp_traje}, where a fine-resolution discretisation of the Stanford Bunny is transported onto a deformed configuration, represented on a coarser mesh.
\begin{figure}
    \centering
    \includegraphics[width=0.30\linewidth]{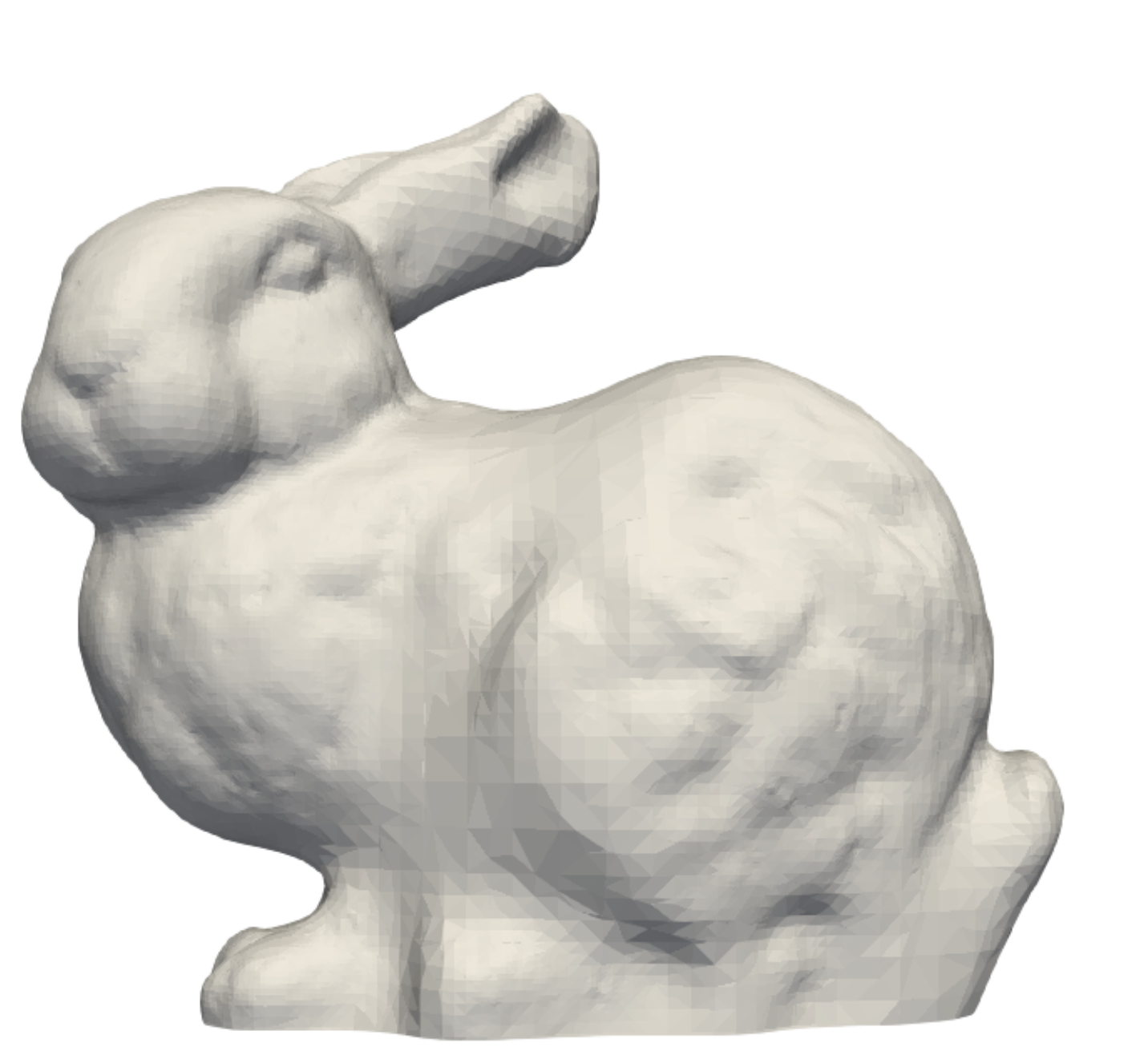}
    \includegraphics[width=0.30\linewidth]{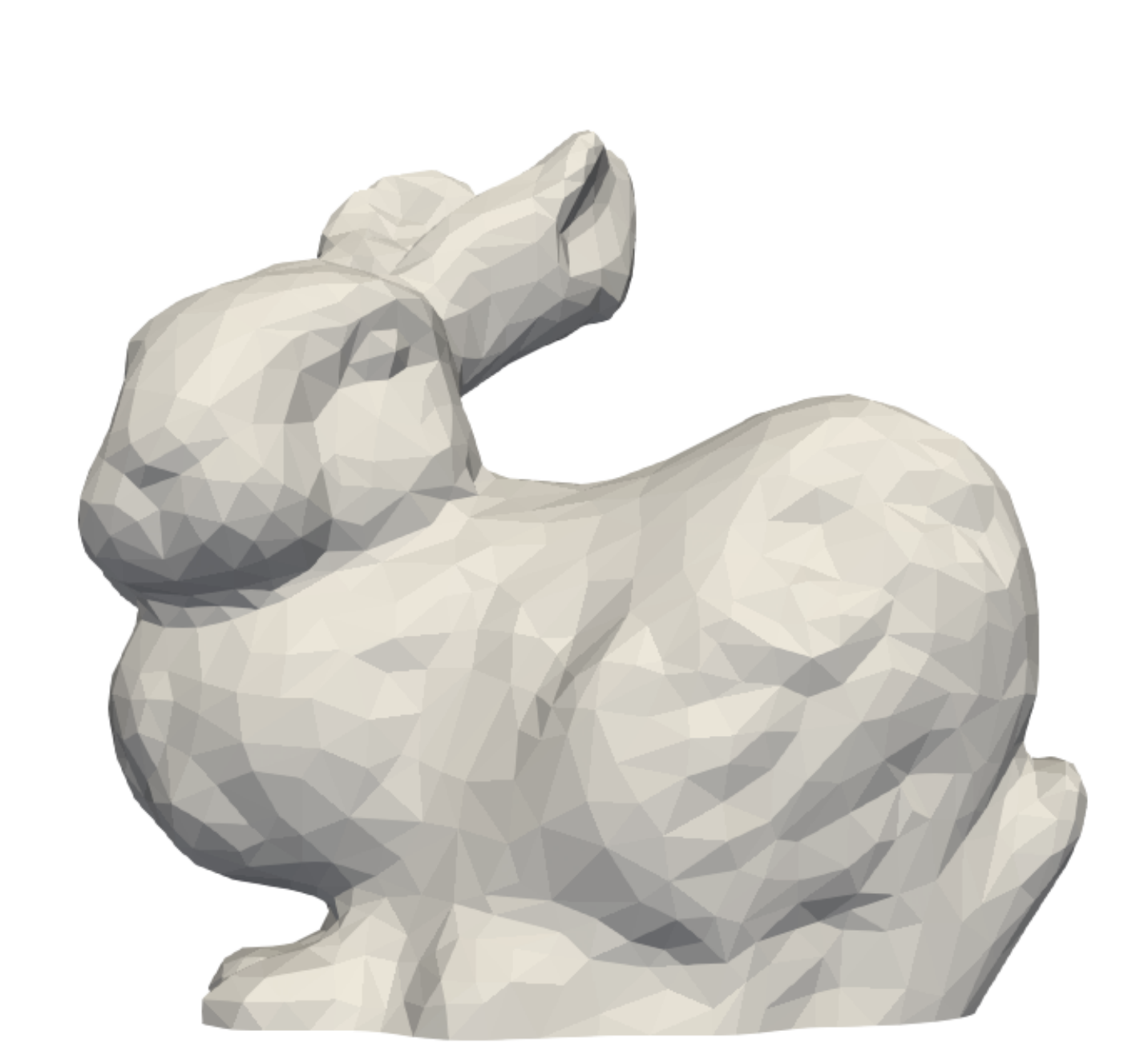}
    \includegraphics[width=0.30\linewidth]{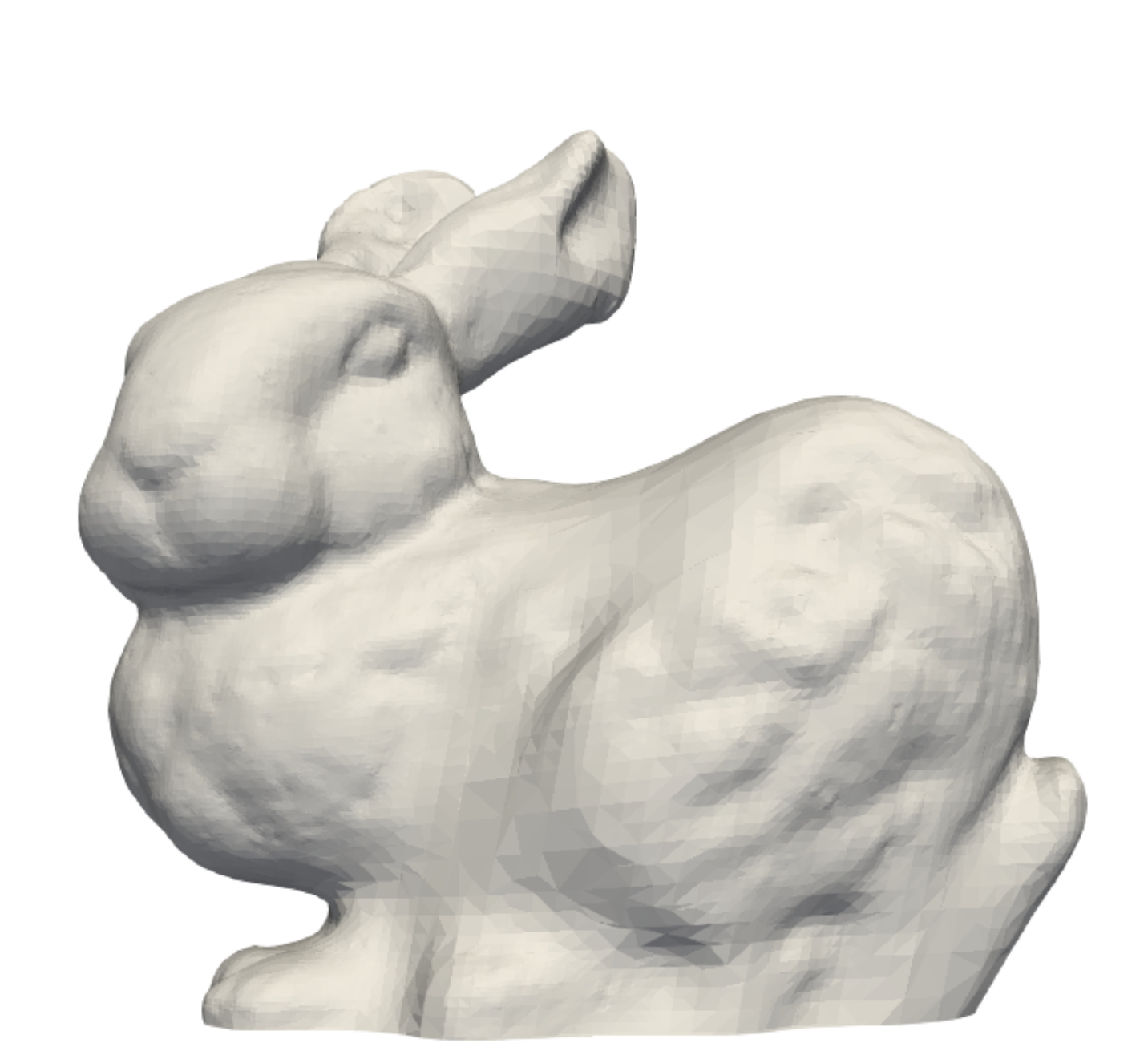}\\
    \includegraphics[width=0.30\linewidth]{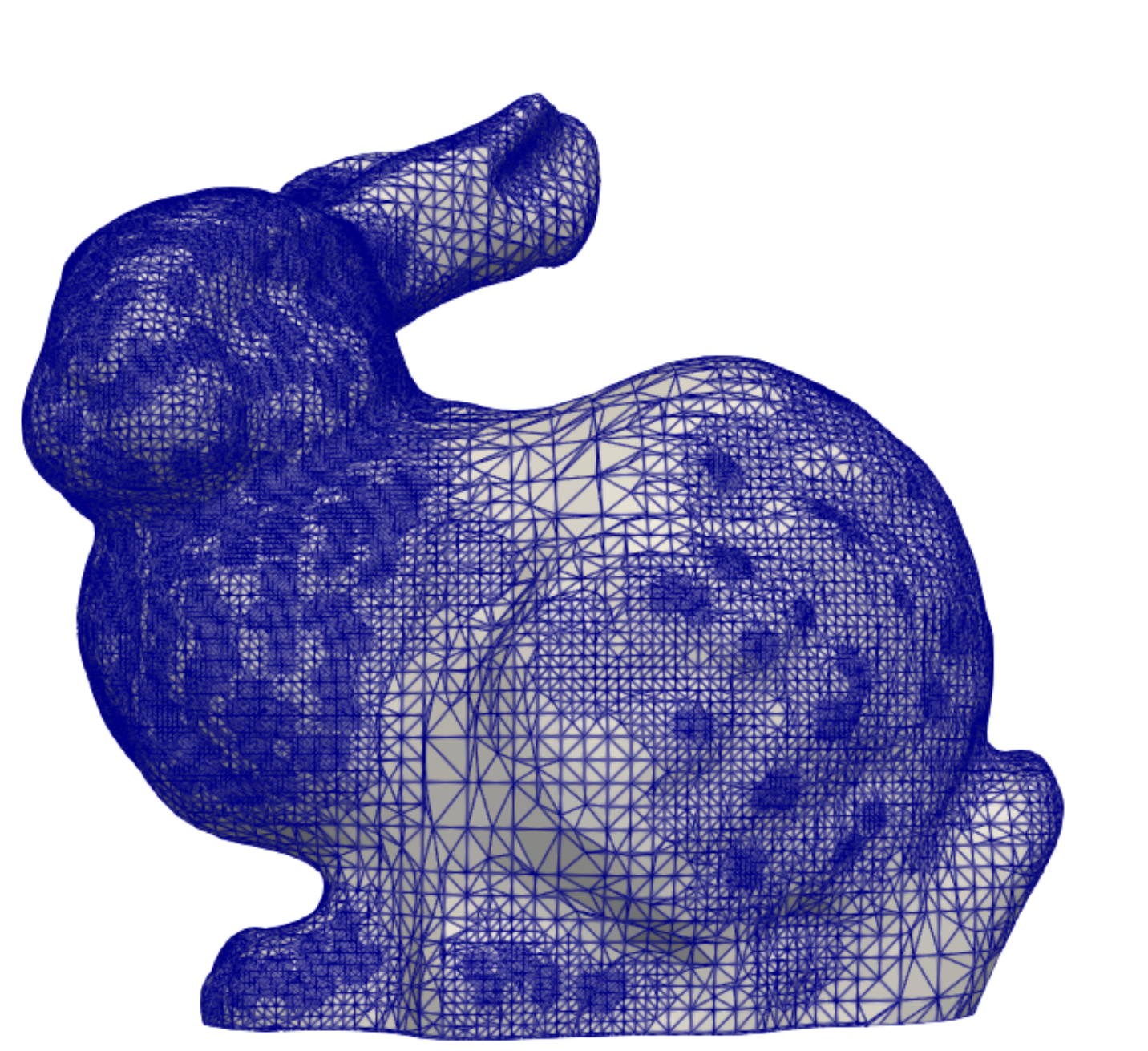}
    \includegraphics[width=0.30\linewidth]{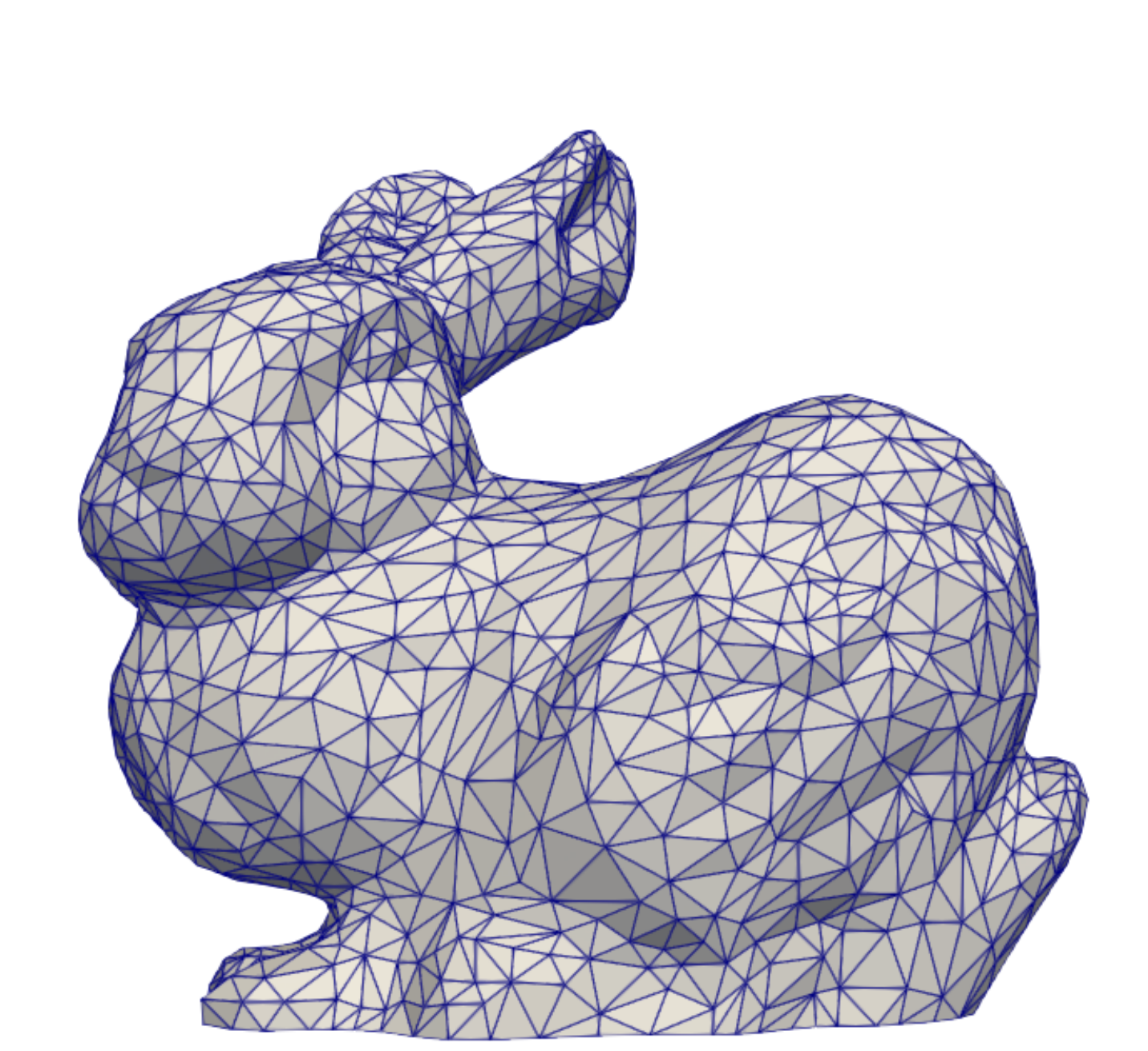}
    \includegraphics[width=0.30\linewidth]{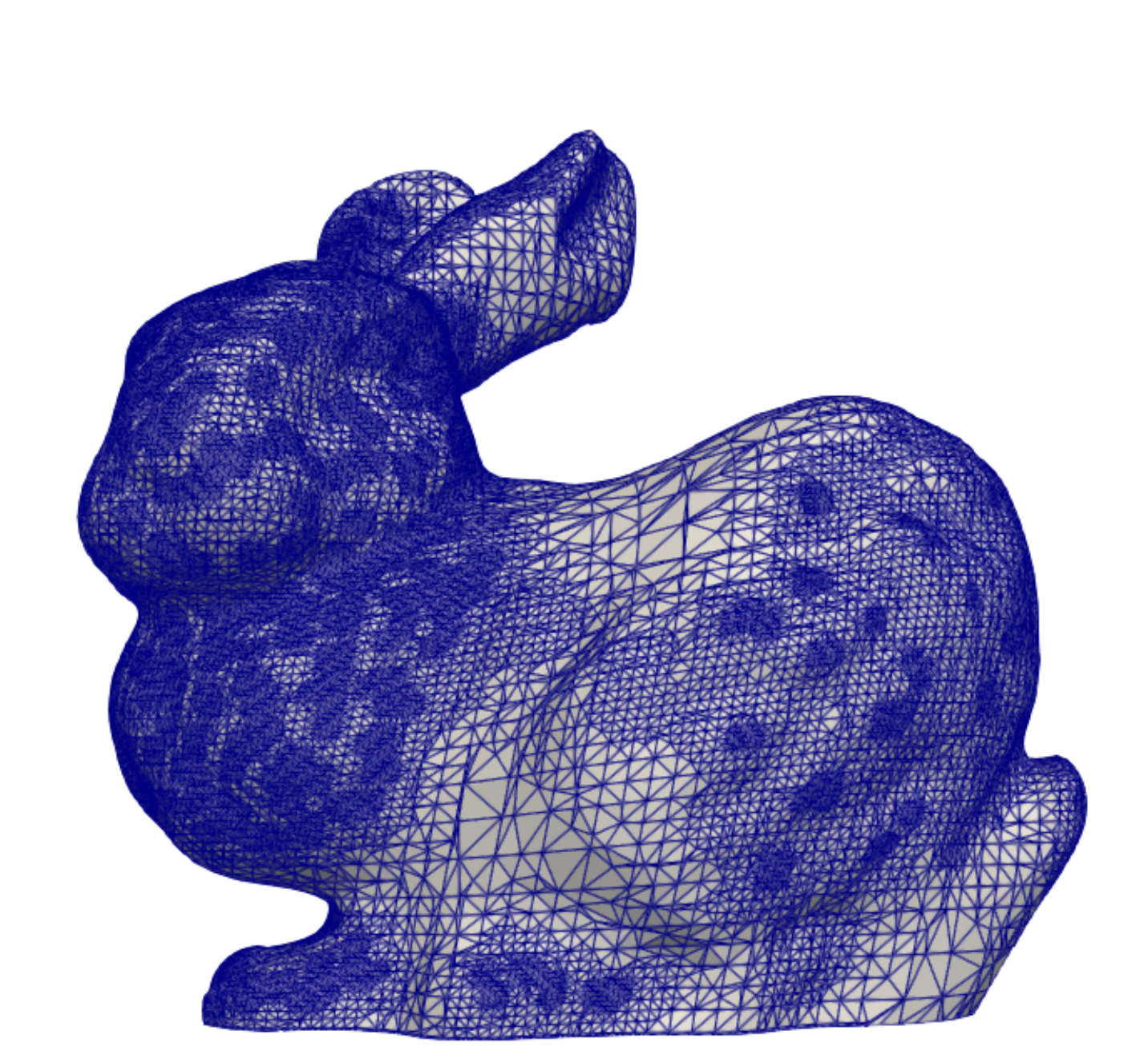}

    \caption{Left: source bunny model exhibiting the topology to be transferred. Center: target bunny model onto which the topology is to be mapped. Right: target bunny model after completion of the topology transfer.}
    \label{fig:1}
\end{figure}

\begin{figure}
    \centering
    \includegraphics[width=0.90\linewidth]{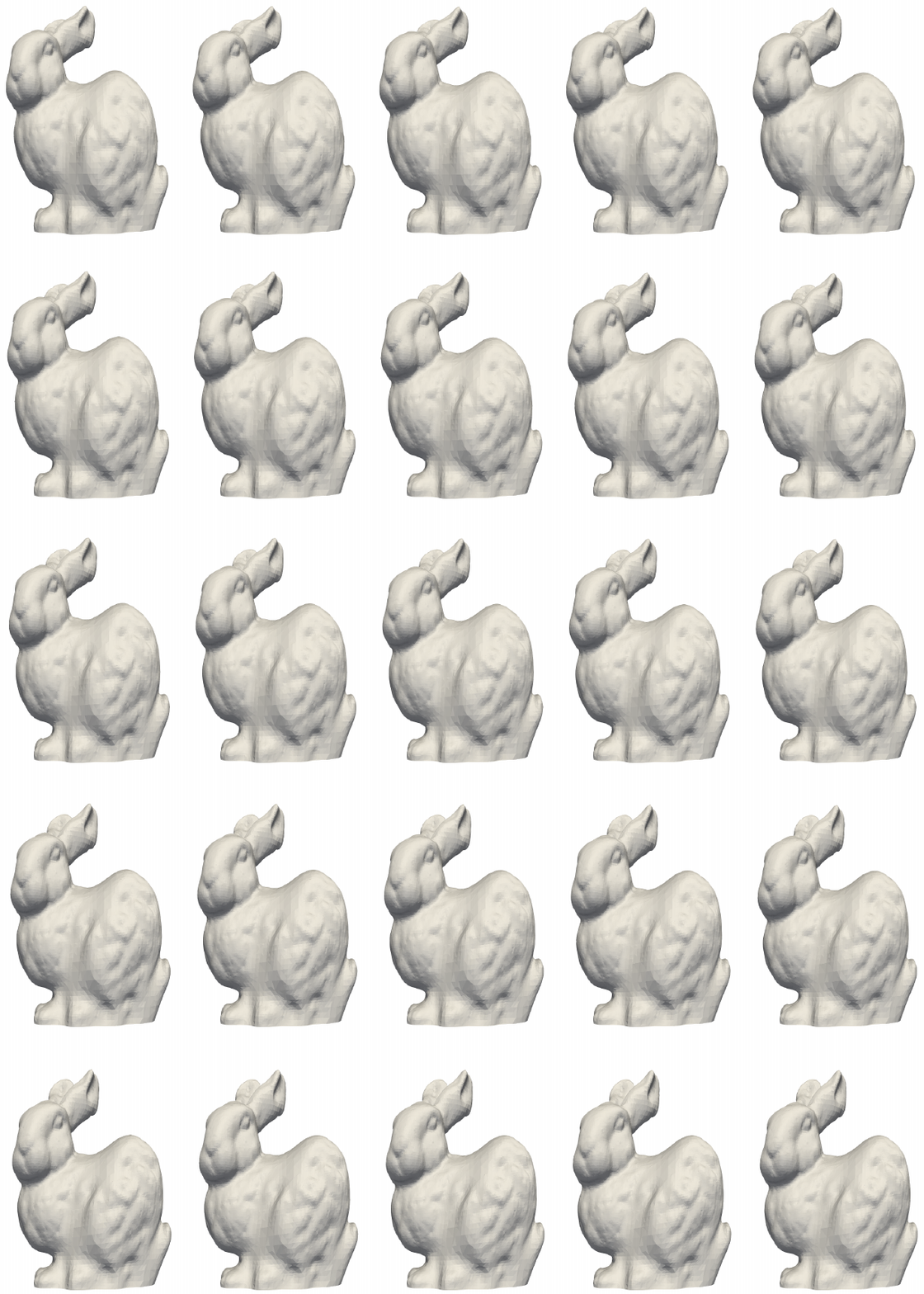}
    \caption{Deformation of the bunny in time.}
    \label{fig:bunny_traje}
\end{figure}

\begin{figure}
    \centering
    \includegraphics[width=1.00\linewidth]{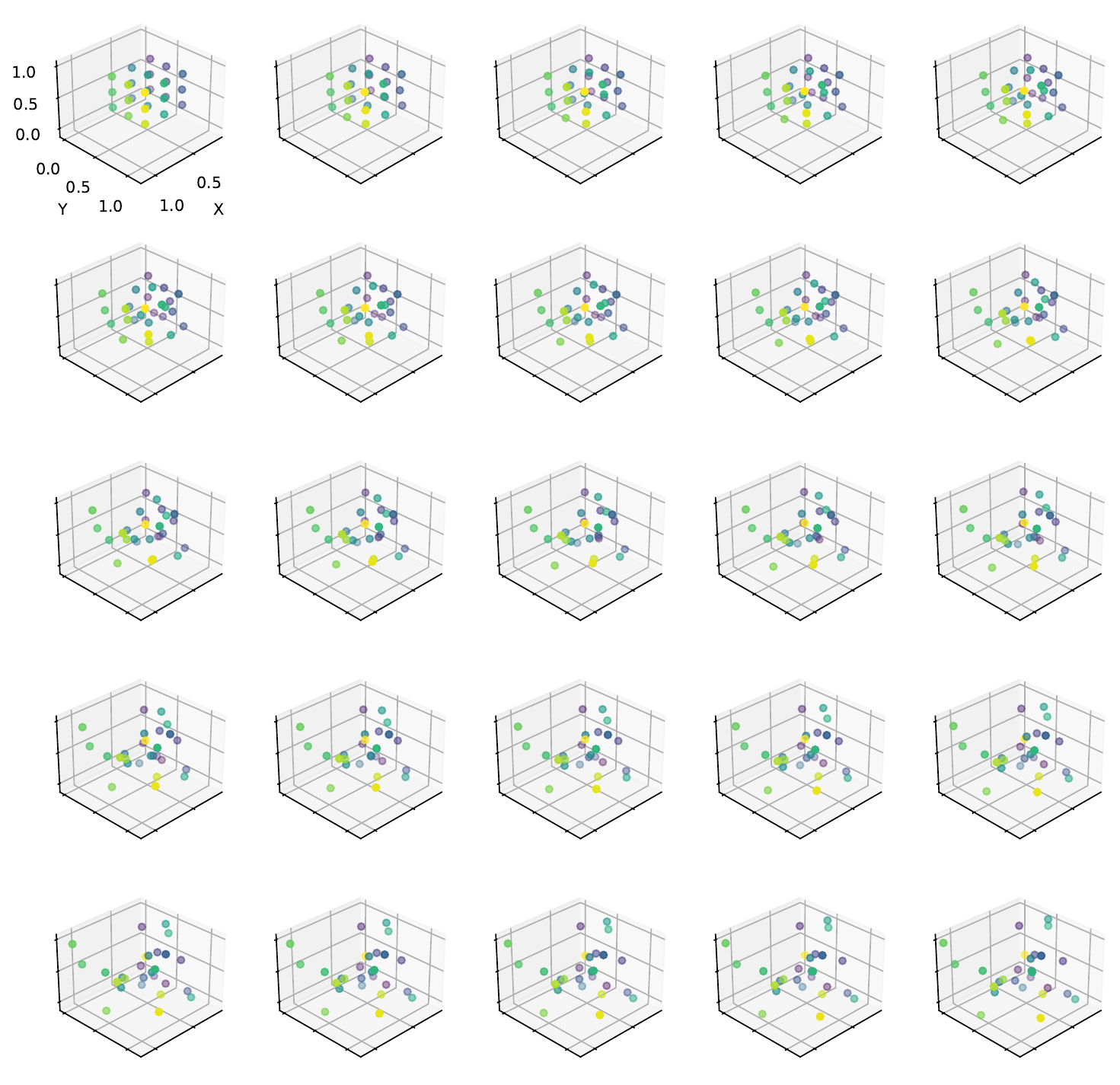}
    \caption{Position of the control points in time.}
    \label{fig:cp_traje}
\end{figure}

\section{Generative models}
\label{sec:genmodel}
In this section, we introduce the generative models methodology for sampling new trajectories of meshes and for performing reduction in parameter space.
As generative models, we use a combination of Proper Orthogonal Decomposition and STARFlow \cite{gu_starflow_2025} in the control points velocity field.\\

We assume that we have $N$ instantaneous velocity fields of control points

\[
\boldsymbol{\alpha}^{(i)}:[0,1]\rightarrow \mathcal{C} \quad i=1\ldots N.
\]
\subsection{Proper Orthogonal Decomposition}
The POD (Proper Orthogonal Decomposition) space is defined as a subspace $\mathbb{V}_{\mathrm{rb}} \subset \mathbb{V}$ of dimension $M < N$ that minimises the mean squared approximation error over a given set of snapshots $\{\boldsymbol{\alpha}^{(i)}\}_{i=1}^N \subset \mathbb{V}$. More precisely, $\mathbb{V}_{\mathrm{rb}}$ is such that
\[
\sqrt{\frac{1}{N} \sum_{i=1}^{N} \inf_{\boldsymbol{\alpha}^{\mathrm{rb}} \in \mathbb{V}_{\mathrm{rb}}} \left\| \boldsymbol{\alpha}^{(i)} - \boldsymbol{\alpha}^{\mathrm{rb}} \right\|_{\mathbb{V}}^2}
\]
is minimised \cite{rozza_reduced_2008}.  
In the following, we consider the Hilbert space $\mathbb{V} = L^{2}([0,1],\mathcal{C})$.

The POD space is constructed by selecting the $M$ eigenfunctions associated to the $M$ largest eigenvalues of the correlation operator, defined as
\[
\mathrm{C}(\boldsymbol{\alpha}) = \frac{1}{N} \sum_{i=1}^{N} \left( \boldsymbol{\alpha}, \boldsymbol{\alpha}^{(i)} \right)_{\mathbb{V}} \, \boldsymbol{\alpha}^{(i)},
\]
where $(\cdot,\cdot)_{\mathbb{V}}$ denotes the inner product in $\mathbb{V}$.

To compute these eigenfunctions, one first assembles the covariance matrix
\[
A_{ij} = \left( \boldsymbol{\alpha}^{(i)}, \boldsymbol{\alpha}^{(j)} \right)_{\mathbb{V}}, \quad i,j = 1,\dots,N,
\]
and computes its eigenpairs $(\lambda^{(i)}, \boldsymbol{v}^{(i)})$, where the eigenvalues $\lambda^{(i)}$ are ordered in descending magnitude. 

An orthonormal basis $\{\boldsymbol{\xi}_{i}\}_{i=1}^{M}$ of the reduced space $\mathbb{V}_{\mathrm{rb}}$ is then given by
\[
\boldsymbol{\xi}_{i} = \frac{1}{\sqrt{\lambda^{(i)}}} \sum_{j=1}^{N} \left( \boldsymbol{v}^{(i)} \cdot \boldsymbol{e}^{(j)} \right) \boldsymbol{\alpha}^{(j)}, 
\qquad i = 1,\dots,M,
\]
where $\{\boldsymbol{e}^{(j)}\}_{j=1}^{N}$ denotes the canonical basis of $\mathbb{R}^{N}$.

Thanks to the orthonormality of the basis $\{\boldsymbol{\xi}_{i}\}_{i=1}^{M}$, the orthogonal projection of any function $\boldsymbol{\alpha} \in \mathbb{V}$ onto $\mathbb{V}_{\mathrm{rb}}$ is given by
\[
\Pi_{\mathbb{V}_{\mathrm{rb}}} \boldsymbol{\alpha} 
= \sum_{i=1}^{M} \left( \boldsymbol{\alpha}, \boldsymbol{\xi}_{i} \right)_{\mathbb{V}} \, \boldsymbol{\xi}_{i}.
\]
The use of Proper Orthogonal Decomposition (POD) in the space of FFD-ODE control points velocities is substantiated by the following theorems. Our objective is to demonstrate that is to show that the mesh reconstruction error obtained from the reconstructed sequence of finite-dimensional control-point instantaneous velocities is both upper and lower bounded by the reconstruction error associated with the reconstructed sequence of these finite-dimensional control-point instantaneous velocities themselves. \\ 
Let $\mathcal{I}=\{0,\dots,m\}\times\{0,\dots,n\}\times\{0,\dots,o\}$ denote the composite index set, and call an index
$c=(i,j,k)\in\mathcal{I}$ a {boundary index} if $i(m-i)\,j(n-j)\,k(o-k)=0$, an {interior index} otherwise; we
write $\mathcal{I}^{\circ}$ for the set of interior indices and
\[
\mathcal{C}_{0}:=\{\boldsymbol{v}\in\mathcal{C}\;:\;
v_{c}=\boldsymbol{0}\ \ \forall\,c\in\mathcal{I}\setminus\mathcal{I}^{\circ}\}.
\]
Then it holds the following theorem.

\begin{theorem}\label{thm:stability}
Let $m,n,o\ge 2$ and $\bar m=\max\{m,n,o\}$. Let
$\boldsymbol{\gamma},\boldsymbol{\delta}\in C([0,1],\mathcal{C})$ take
values in $\mathcal{C}_{0}$, let $\boldsymbol{x}^{(0)}\in[0,1]^{3}$,
and let $\boldsymbol{y}_{\boldsymbol{\gamma}},
\boldsymbol{y}_{\boldsymbol{\delta}}$ denote the corresponding
solutions of
\[
\dot{\boldsymbol{y}}(t)=\boldsymbol{F}_{\boldsymbol{\alpha}(t)}
\bigl(\boldsymbol{y}(t)\bigr),
\qquad \boldsymbol{y}(0)=\boldsymbol{x}^{(0)},
\qquad \boldsymbol{\alpha}\in\{\boldsymbol{\gamma},\boldsymbol{\delta}\}.\]
 Define
\[
A_{1}:=2\sqrt3\,\bar m\,\|\boldsymbol{\gamma}\|_{L^{1}([0,1],\mathcal{C})},
\qquad
A_{2}:=2\sqrt3\,\bar m\,\|\boldsymbol{\gamma}\|_{L^{2}([0,1],\mathcal{C})},
\]
so that $A_{1}\le A_{2}$. Then, for all $t\in[0,1]$,
\[
\|\boldsymbol{y}_{\boldsymbol{\delta}}(t)
 -\boldsymbol{y}_{\boldsymbol{\gamma}}(t)\|_{1}
\;\le\;
\sqrt{\tfrac{3}{8}}\;
\exp(A_{1})\,
\|\boldsymbol{\delta}-\boldsymbol{\gamma}\|_{\mathbb{V}},
\]
and
\[
\|\boldsymbol{y}_{\boldsymbol{\delta}}
 -\boldsymbol{y}_{\boldsymbol{\gamma}}\|_{H^{1}([0,1],\mathbb{R}^{3})}
\;\le\;
\sqrt{\tfrac{3}{8}}\,
\sqrt{\exp(2A_{1})+\bigl(1+A_{2}\exp(A_{1})\bigr)^{2}}
\;\|\boldsymbol{\delta}-\boldsymbol{\gamma}\|_{\mathbb{V}} .
\]
\end{theorem}

We can further generalise this framework by simultaneously considering multiple trajectories. Let
\[
\begin{cases}
\dot{\boldsymbol{y}}_{\boldsymbol{\delta}}(t,\boldsymbol{x})=\boldsymbol{F}_{\boldsymbol{\delta}(t)}\big(\boldsymbol{y}_{\boldsymbol{\delta}}(t,\boldsymbol{x})\big),\\[4pt]
\boldsymbol{y}_{\boldsymbol{\delta}}(0,\boldsymbol{x})=\boldsymbol{x},
\end{cases}
\]
and
\[
\begin{cases}
\dot{\boldsymbol{y}}_{\boldsymbol{\gamma}}(t,\boldsymbol{x})=\boldsymbol{F}_{\boldsymbol{\gamma}(t)}\big(\boldsymbol{y}_{\boldsymbol{\gamma}}(t,\boldsymbol{x})\big),\\[4pt]
\boldsymbol{y}_{\boldsymbol{\gamma}}(0,\boldsymbol{x})=\boldsymbol{x}.
\end{cases}
\]
By integrating in the spatial variable and invoking the previous theorem, we define the function space
\[
\mathbb{W}=\Big\{\boldsymbol{f}:[0,1]\times[0,1]^{3}\to\mathbb{R}^{3}\,\big|\; (\boldsymbol{f},\boldsymbol{f})_{\mathbb{W}}<+\infty\Big\},
\]
where the inner product is given by
\[
(\boldsymbol{f},\boldsymbol{g})_{\mathbb{W}}
=\int_{[0,1]^{3}}\int_{0}^{1}\big(\boldsymbol{f}(t,\boldsymbol{x}),\boldsymbol{g}(t,\boldsymbol{x})\big)_{\mathbb{R}^{3}}
+\left(\frac{\partial\boldsymbol{f}(t,\boldsymbol{x})}{\partial t},\frac{\partial\boldsymbol{g}(t,\boldsymbol{x})}{\partial t}\right)_{\mathbb{R}^{3}}\,dt\,d\boldsymbol{x}.
\]
Under these definitions, we obtain the estimate
\[
\|\boldsymbol{y}_{\boldsymbol{\delta}}-\boldsymbol{y}_{\boldsymbol{\gamma}}\|_{\mathbb{W}}
\leq \sqrt{1+2A e^{A}+\big(A^{2}+1\big)e^{2A}}\;\|\boldsymbol{\gamma}-\boldsymbol{\delta}\|_{\mathbb{V}}.
\]
It also holds a reverse inequality.
\begin{theorem}\label{thm:lower}
Let $m,n,o\ge 2$, $\bar{m}=\max\{m,n,o\}$, and let
$\boldsymbol{\alpha},\boldsymbol{\gamma}\in C([0,1],\mathcal{C}_{0})$.
Then $\boldsymbol{y}_{\boldsymbol{\alpha}},
\boldsymbol{y}_{\boldsymbol{\gamma}}\in\mathbb{W}$ and
\[
\|\boldsymbol{y}_{\boldsymbol{\alpha}}
-\boldsymbol{y}_{\boldsymbol{\gamma}}\|_{\mathbb{W}}
\;\ge\;
\frac{\omega_{\boldsymbol{\gamma}}}
{\sqrt{1+12\,\bar{m}^{2}\,
\|\boldsymbol{\alpha}\|_{L^{\infty}([0,1],\mathcal{C})}^{2}}}\;
\|\boldsymbol{\alpha}-\boldsymbol{\gamma}\|_{\mathbb{V}},
\]
with $\omega_{\boldsymbol{\gamma}}=\frac{\exp \left(-2 \sqrt{3} \bar{m}\|\boldsymbol{\gamma}\|_{L^1([0,1], c)}\right)}{(2 m+1)\binom{2 m}{m}(2 n+1)\binom{2 n}{n}(2 o+1)\binom{2 o}{o}}$. 
\end{theorem}

This result implies that there is a bijection between a control-points trajectory $\boldsymbol{\alpha}\in \mathbb{V}$ and its corresponding solution $\boldsymbol{y}\in \mathbb{W}$. \\ Consequently, the use of POD on the control-point instantaneous velocities does not introduce spurious solutions and allows our methodology to remain independent of the specific topology of each target mesh.

\subsection{STARFlow}
A generative model is a (potentially not absolutely continuous) parametric probability distribution that allows efficient sampling. In this work, we focus on normalising flows \cite{rezende_variational_2015}, which represent the target distribution as
\[
p_{Z}(\boldsymbol{T}_{\boldsymbol{\zeta}}(\boldsymbol{x}))\left|\operatorname{det}\left(\nabla\boldsymbol{T}_{\boldsymbol{\zeta}}(\boldsymbol{x})\right)\right|,
\]
where $\boldsymbol{T}_{\boldsymbol{\zeta}}$ is an invertible neural network with a diagonal or triangular Jacobian, parameterised by $\boldsymbol{\zeta}$, and $p_{Z}$ is a base distribution that admits efficient sampling. A sample from the generative model is produced by first drawing a sample from $Z$ and subsequently applying the transformation $\boldsymbol{T}_{\boldsymbol{\zeta}}$.
In this study, we employ the AutoEncoder–Normalising Flow architecture STARFlow \cite{gu_starflow_2025} to generate new vectors of POD coefficients. Specifically, the POD coefficient vectors serve as inputs to an AutoEncoder composed of a three-layer encoder, $\operatorname{Enc}_{\boldsymbol{\theta}}:\mathbb{R}^{M}\rightarrow\mathbb{R}^{
D}$, and a three-layer decoder, $\operatorname{Dec}_{\boldsymbol{\sigma}}:\mathbb{R}^{D}\rightarrow\mathbb{R}^{M}$. Both networks are configured with a hidden layer width of 500, ReLU activation functions, batch normalisation, and a latent space of dimension $D$. The trainable parameters of the encoder and decoder are denoted by $\boldsymbol{\theta}$ and $\boldsymbol{\sigma}$, respectively.
The network parameters are optimised in an end-to-end fashion by solving
\[
(\boldsymbol{\theta}^{\star},\boldsymbol{\sigma}^{\star})=\operatorname{arg}\min\limits_{\boldsymbol{\theta},\boldsymbol{\sigma}}\sum_{i=1}^{N}\left \|\left(\operatorname{id}-\operatorname{Dec}_{\boldsymbol{\sigma}}\circ \operatorname{Enc}_{\boldsymbol{\theta}}\right)\left(\sum_{j=1}^{M}(\boldsymbol{\alpha}^{(i)},\boldsymbol{\xi}^{(j)})_{\mathbb{V}}\boldsymbol{e}^{(j)}\right)\right\|^{2},
\]
where \(\operatorname{Enc}_{\boldsymbol{\theta}}\) and \(\operatorname{Dec}_{\boldsymbol{\sigma}}\) denote the encoder and decoder parametrised by \(\boldsymbol{\theta}\) and \(\boldsymbol{\sigma}\), respectively, and the loss function is given by the squared reconstruction error.

The resulting latent representation of the POD coefficients, as inferred by this AutoEncoder, is then utilised to train TARFlow \cite{zhai_normalizing_2025}, a probabilistic generative model based on transformer-driven autoregressive normalising flows, which in our notation is $\boldsymbol{T}_{\boldsymbol{\zeta}}$. A rigorous construction of $\boldsymbol{T}_{\boldsymbol{\gamma}}$ is provided in the appendix. As base distribution, we adopt the Standard Gaussian one.

A schematic overview of the training pipeline is reported in Fig.~\ref{fig:train_pipeline}.

\begin{figure}
    \centering
     \includegraphics[width=1.0\linewidth]{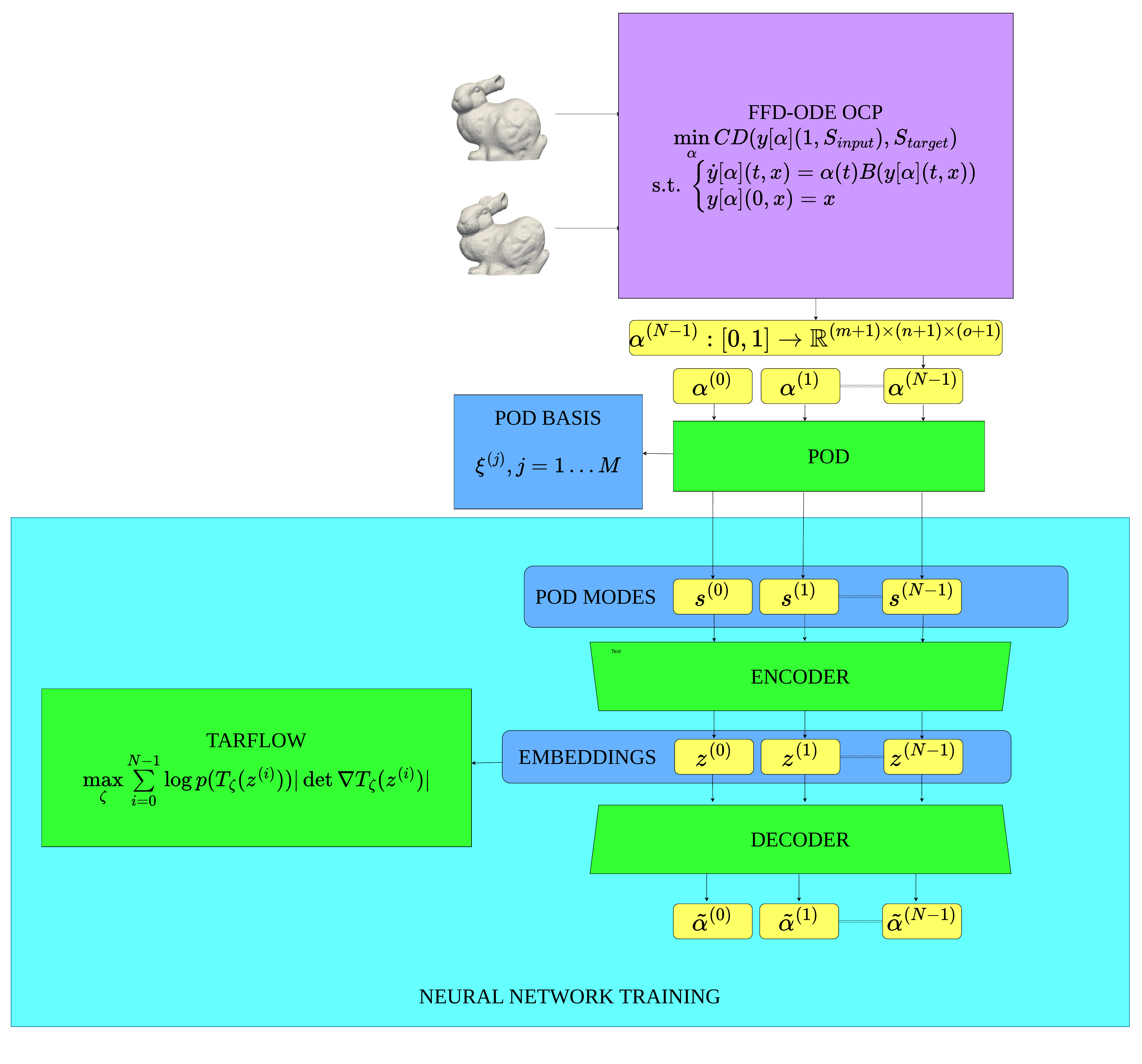}
    \caption{Summary of the training pipeline.}
    \label{fig:train_pipeline}
\end{figure}

\subsection{Mesh dynamic generation and ROM assembly}
In this section, we concisely present the procedure for generating novel mesh evolutions after completion of the training phase, with particular emphasis on its use in the construction of ROMs. The procedure is given by:
\begin{itemize}
    \item Sample $\boldsymbol{\nu}$ from the prescribed Gaussian distribution,
    \item Compute a new latent embedding $\tilde{\boldsymbol{z}}=\boldsymbol{T}_{\gamma}^{-1}(\boldsymbol{\nu})$,
    \item Compute a new POD coefficient $\tilde{\boldsymbol{s}}=\operatorname{Dec}_{\boldsymbol{\sigma}}(\tilde{\boldsymbol{z}})$,
    \item Reconstruct the control-point velocity field  $\tilde{\boldsymbol{\alpha}}=\sum\limits_{i=1}^{M}\tilde{\boldsymbol{s}}_{i}\boldsymbol{\xi}^{(i)}$,
    \item Obtain a new mesh evolution by applying $\boldsymbol{G}_{\tilde{\boldsymbol{\alpha}}}$ to the reference mesh.
\end{itemize}

The above procedure enables efficient sampling of physically consistent mesh evolutions. Once a mesh-dynamics model has been specified, numerical simulations can be carried out on the generated meshes to evaluate a prescribed quantity of interest. By iteratively applying the sampling process, one obtains a dataset of quantities of interest associated with distinct latent embeddings. This dataset is subsequently employed to construct a non-intrusive ROM, which enables computationally efficient shape optimisation of the quantity of interest, owing to the low-dimensional parametrisation of the embedding space. The ROM construction workflow is summarised in Fig.~\ref{fig:assembly_pipeline}.
\begin{figure}
    \centering
    \includegraphics[width=0.8\linewidth]{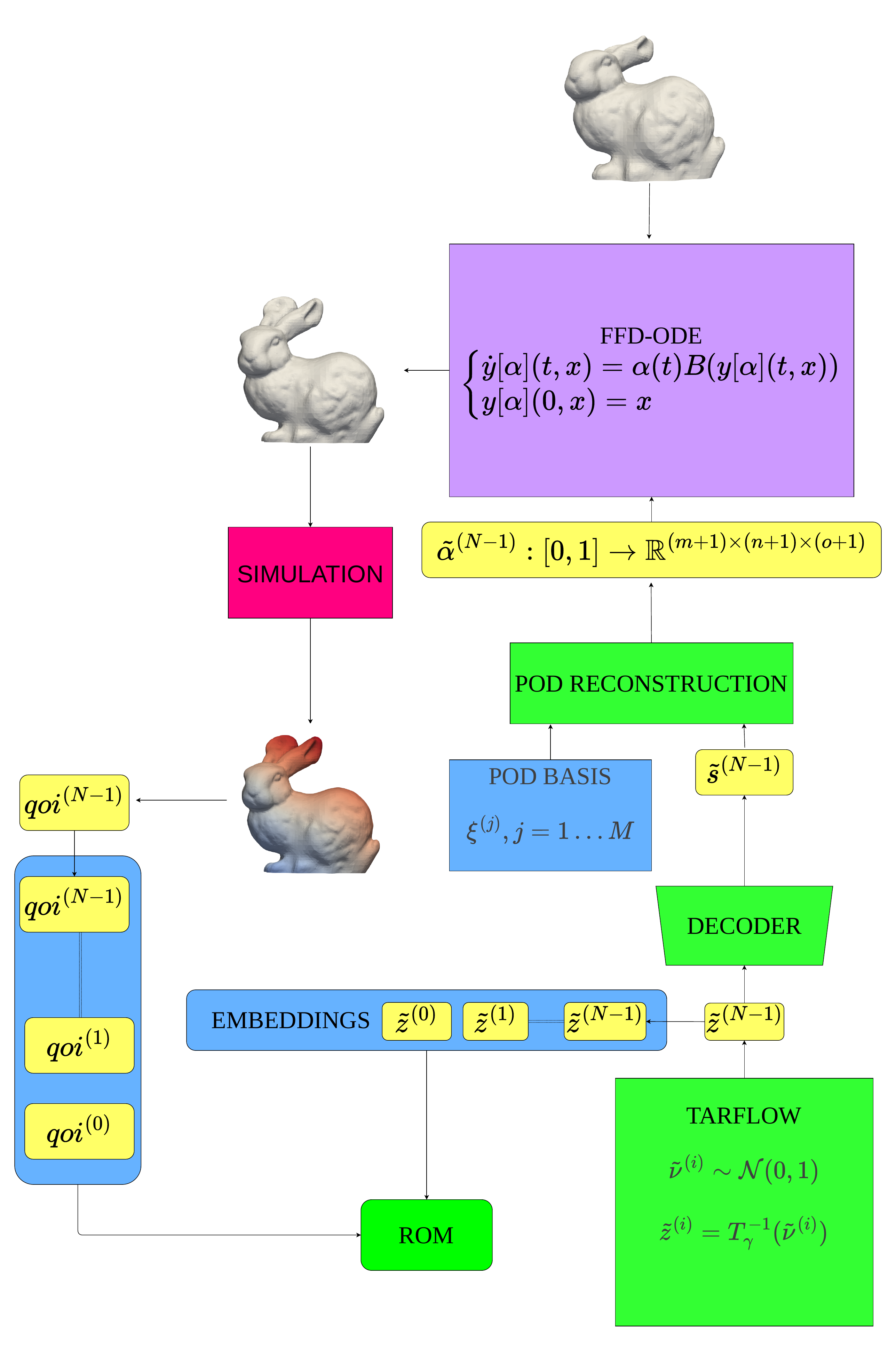}
    \caption{Workflow for the construction of the ROM by generating snapshots using STARFlow and FFD-ODE.}
    \label{fig:assembly_pipeline}
\end{figure}

\section{Test cases}
\label{sec:numerics}
We consider two benchmark problems:
\begin{itemize}
    \item an incompressible fluid flow induced by the motion of a Stanford Bunny computational domain. As the quantity of interest, we consider the spatially averaged velocity magnitude;
    \item a hypersonic flow over a deforming sphere. As the quantity of interest, we consider the average drag coefficient.
\end{itemize}
The trajectories $\boldsymbol{\alpha}_{i}$ are discretised using 102 time steps for the ROM construction. The number of control points is set to 4 along each spatial axis. Furthermore, in both cases we take as latent dimensionality for the generative models $D=10$.
Our objective is to compare the Leave-One-Out Cross-Validation (LOOCV) discretised $\ell^{1}$ and $\ell^{2}$ norms of the errors associated with ROMs based on FFD-ODE and on STARFlow. The errors are evaluated on normalised outputs, and all inputs are standardised prior to inference. Let $q \in \mathbb{R}^{N}$ and $\tilde{q} \in \mathbb{R}^{N}$ denote, respectively, the normalised reference values and the normalised LOOCV estimates of the quantity of interest. The discretised $\ell^{1}$ error is defined as
\[
\frac{1}{N}\sum_{i=1}^{N}\left|q_{i}-\tilde{q}_{i}\right|,
\]
while the discretised $\ell^{2}$ error is given by
\[
\sqrt{\frac{1}{N}\sum_{i=1}^{N}\left|q_{i}-\tilde{q}_{i}\right|^{2}}.
\]

Generative models are employed to achieve dimensionality reduction in the parameter space, since the performance of non-intrusive ROMs is strongly influenced by the number of parameters. Compared to classical parameter-space reduction techniques, generative models offer the additional capability of directly sampling new geometries.

\subsection{Flow generated by a Stanford Bunny}
We consider the fluid flow induced by the motion of a Stanford Bunny whose base remains fixed. The bunny is embedded in a rectangular box $B$ of dimensions $[0.2, 0.9] \times [0.2, 0.77] \times [0, 0.69]$ meters, with its base lying on the plane $z = 0$. The bunny occupies approximately one-third of the volume of $B$. The governing equations are the incompressible Navier–Stokes equations, which in this setting read:
\[
\begin{cases}
\displaystyle \frac{\partial\boldsymbol{u}}{\partial t}(t,\boldsymbol{x}) + (\boldsymbol{u}\cdot \nabla)\boldsymbol{u} - \mu \Delta \boldsymbol{u} + \nabla p = 0, & \boldsymbol{x}\in \Omega(t),\ t \in (0,1],\\[0.4em]
\nabla \cdot \boldsymbol{u}(t, \boldsymbol{x}) = 0, & \boldsymbol{x}\in \Omega(t),\ t \in (0,1],\\[0.4em]
\boldsymbol{u}(t,\boldsymbol{x}) =\boldsymbol{F}_{\boldsymbol{\alpha}(t)}(\boldsymbol{x}), & \boldsymbol{x} \in \partial \Omega(t),\ t \in [0,1],\\
\boldsymbol{u}(0,\boldsymbol{x})=0 & \boldsymbol{x} \in \partial \Omega(0)
\end{cases}
\]
where $\boldsymbol{u}$ denotes the velocity field, $p$ the pressure, $\mu$ the kinematic viscosity, $\Omega(t)$ the time–dependent fluid domain which moves with velocity $\boldsymbol{F}_{\boldsymbol{\alpha}(t)}(\boldsymbol{x})$, which is the FFD-ODE model defined in the previous section. We remark that the domain remains in the box as $\boldsymbol{F}_{\boldsymbol{\alpha}}$ restricted to the box boundary is 0.\\
To numerically compute the solution, we employ an Arbitrary Lagrangian–Eulerian (ALE) formulation on the reference Stanford Bunny geometry. We set $\mu = 0.01$ $m^{2}/s$ and adopt a reference velocity of $0.5$ $m/s$ during the deformation. The resulting partial differential equations are discretised using a finite element method with a pressure stabilisation term \cite{richter_fluid-structure_2017}.

The weak ALE formulation consists of the following steps:
\[
\text{Find } p_{h}\in Q_{h},\ \boldsymbol{u}_{h}\in V_{h} \text{ such that}
\]
\[
\begin{cases}
\begin{aligned}
& \int_{\Omega_0} J_\alpha\left(\Big(\mathbf{H}_\alpha^{-1}\Big(\boldsymbol{u}_{h}-\widehat{\boldsymbol{F}}_{\boldsymbol{\alpha}}\Big)\Big) \cdot \nabla\right) \boldsymbol{u}_{h} \cdot \boldsymbol{w}_{h}\, d \boldsymbol{x} \\
&\quad + \int_{\Omega_0} J_\alpha\left(\left(p_h I + \mu\left(\nabla \boldsymbol{u}_{h}\,\mathbf{H}_\alpha^{-1} + \mathbf{H}_\alpha^{-T}\left(\nabla \boldsymbol{u}_{h}\right)^T\right)\right)\mathbf{H}_\alpha^{-T}\right) : \nabla \boldsymbol{w}_{h}\, d \boldsymbol{x} \\
&\quad + \int_{\Omega_0} \left(\operatorname{div}\left(J_\alpha \mathbf{H}_\alpha^{-T} \boldsymbol{u}_{h}\right)\right) q_h\, d \boldsymbol{x} 
+ \int_{\Omega_0} J_\alpha\, \partial_t \boldsymbol{u}_{h} \cdot \boldsymbol{w}_{h}\, d \boldsymbol{x} \\
&\quad + \alpha_{\text{stab}}(\boldsymbol{x}, t) \int_{\Omega_0} \nabla p_h \cdot \nabla q_h\, d \boldsymbol{x} = 0 
\quad \forall\, q_{h}\in Q_{h},\ \boldsymbol{w}_{h}\in V_{h}, \quad \forall  t\in [0,1]
\end{aligned}\\[0.4em]
\gamma_{\star}(\boldsymbol{u}_{h}) = \mathcal{P}\widehat{\boldsymbol{F}}_{\boldsymbol{\alpha}},  \quad \forall t\in [0,1]\\
\boldsymbol{u}_{h}(0)=0
\end{cases}
\]
where $\Omega_{0}$ is the reference configuration,
\[
\begin{cases} 
\partial_t \boldsymbol{A}_\alpha(X, t)=F_{\boldsymbol{\alpha}(t)}\left(\boldsymbol{A}_{\boldsymbol{\alpha}}(X, t)\right), \\
\boldsymbol{A}_{\boldsymbol{\alpha}}(X, 0)=X
\end{cases}
\]
is the ALE map, 
$\mathbf{H}_{\boldsymbol{\alpha}} = \nabla \boldsymbol{A}_{\boldsymbol{\alpha}}$ is the deformation gradient, $J_{\boldsymbol{\alpha}} = \det \mathbf{H}_{\boldsymbol{\alpha}}$ is the Jacobian determinant, $\widehat{\boldsymbol{F}}_{\boldsymbol{\alpha}}(X, t):=\boldsymbol{F}_{\boldsymbol{\alpha}(t)}\left(\boldsymbol{A}_{\boldsymbol{\alpha}}(X, t)\right)$ is the remapped velocity field, $\gamma_{\star}$ 
denotes the trace operator, and
\[
\alpha_{\text{stab}}(\boldsymbol{x},t) = 10\left(\mu \frac{J_{\alpha}^{2/3}}{\Delta x^{2}} + |\boldsymbol{u}_{h}| \frac{J_{\alpha}^{1/3}}{\Delta x} \right),
\]
is a stabilisation term, where $\Delta x$ denotes the local mesh cell diameter. The discrete spaces $V_{h}$ and $Q_{h}$ are selected as a coupled quadratic (P2-P2) finite element spaces for the velocity and pressure fields, respectively. This choice constitutes an inf-sup unstable velocity–pressure pairing and therefore necessitates the inclusion of the aforementioned stabilisation term.
The computational mesh consists of 14,491 vertices and 67,872 cells (see Fig. \ref{fig:2}). For temporal discretisation, we employ a linearly implicit Euler scheme.

The system is solved for 100 distinct $5 \times 5 \times 5$ velocity fields of control points, denoted $\boldsymbol{\alpha}^{(i)}$, which are generated by deforming the reference Stanford Bunny via FFD with a $7 \times 7 \times 7$ grid of Gaussian–distributed control points, and then by applying FFD-ODE. All simulations are performed using FEniCSx \cite{baratta_dolfinx_2023, scroggs_basix_2022, alnaes_unified_2014}. Additional 100 velocity fields are produced using generative models. The performance of these generative models is evaluated in terms of the distribution of the average velocity norm and the root mean squared control point velocity field. For training the generative model, we retain the first $M=80$ POD modes.

Since solving the full–order PDE for each velocity field is computationally expensive, we investigate several reduced–order regression models: K–Nearest Neighbours Regression (KNN), Gaussian Process Regression (GPR), and Random Forest Regression (RF) both on the original control point velocity fields and on the latent space representations obtained from the generative models. In all cases, the dependent variable is chosen as the average velocity norm. All models are implemented via Scikit–Learn \cite{pedregosa_scikit-learn_2011}) with default hyperparameters.

\begin{figure}[ht]
    \centering
    \includegraphics[width=1\linewidth]{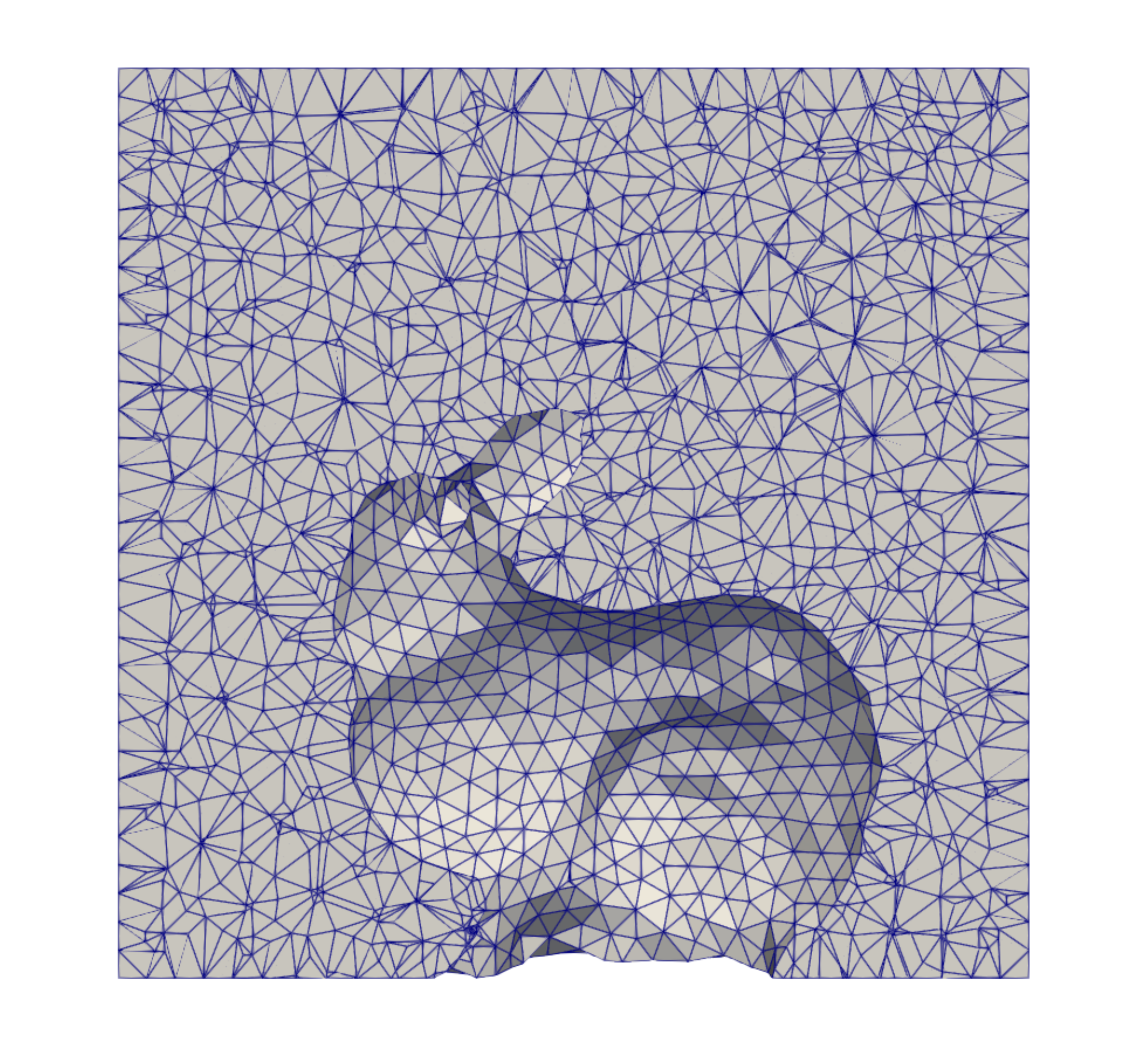}
    \caption{Section of the computational mesh used for the Stanford Bunny test case.}
    \label{fig:2}
\end{figure}

\subsection{Hypersonic flow past a deforming sphere}
We consider a turbulent hypersonic flow past a deforming sphere of radius $1$, centred at the origin and embedded in the cubic computational domain $[-40,40]^3$. The flow is simulated with the OpenFOAM \cite{weller_tensorial_1998} density-based solver \texttt{rhoPimpleFoam}, employing a Spalart–Allmaras RANS turbulence model and a Sutherland-type temperature-dependent viscosity model. The governing equations are the compressible Navier–Stokes equations for a calorically perfect gas,
\[
\begin{cases}
\displaystyle \frac{\partial \rho}{\partial t}+\nabla \cdot(\rho \boldsymbol{u})=0,\\[0.8ex]
\displaystyle \frac{\partial(\rho \boldsymbol{u})}{\partial t}+\nabla \cdot(\rho \boldsymbol{u} \otimes \boldsymbol{u})=-\nabla p+\nabla \cdot \boldsymbol{\tau},\\[0.8ex]
\displaystyle \boldsymbol{\tau}=\mu(T)\left(\nabla \boldsymbol{u}+(\nabla \boldsymbol{u})^T\right)-\frac{2}{3} \mu(T)(\nabla \cdot \boldsymbol{u}) \mathbf{I},\\[0.8ex]
\displaystyle p=\rho R T,\\[0.8ex]
\displaystyle e=C_v T,\\[0.8ex]
\displaystyle \frac{\partial(\rho e)}{\partial t}+\nabla \cdot(\rho \boldsymbol{u} e)+p \nabla \cdot \boldsymbol{u}=\nabla \cdot(\kappa(T) \nabla T)+\boldsymbol{\tau}: \nabla \boldsymbol{u},\\[0.8ex]
\displaystyle \mu(T)=A_s \frac{T^{3 / 2}}{T+T_s},\\[0.8ex]
\displaystyle \kappa(T)=\mu(T) C_v\left(1.32+1.77 \frac{R}{C_v}\right),
\end{cases}
\]
which are integrated over a physical time interval of $0.004\,\mathrm{s}$ (the FFD-ODE is scaled consistently with this time horizon). The thermophysical parameters are set to
$C_{v}=719.302\,\mathrm{J}\,\mathrm{kg}^{-1}\,\mathrm{K}^{-1}$, $A_{s}=1.4792 \times 10^{-6}$, $T_{s}=116\,\mathrm{K}$, and $R=287\,\mathrm{J}\,(\mathrm{kg}\cdot\mathrm{K})^{-1}$.
We have a Dirichlet boundary condition on the moving sphere equal to the sphere velocity, while a Mach 9 is imposed at the box boundary with a \texttt{freeStream} boundary condition.\\
As a consequence, the Reynolds number is approximately $\operatorname{Re} \approx 2 \times 10^8$, which is characteristic of a fully turbulent flow regime. The initial turbulent viscosity is set to $\nu_{0}=10^{-5}\,\mathrm{m^{2}\,s^{-1}}$.\\
The computational mesh comprises 495\,200 grid points (see Fig.~\ref{fig:3}).

The governing equations are solved for 100 distinct $5\times5\times5$ control-point velocity fields $\boldsymbol{\alpha}^{(i)}$, which are generated by deforming a reference spherical geometry via FFD using a $7\times7\times7$ grid of control points, and then by applying FFD-ODE. Generative models are subsequently employed to synthesise an additional 100 velocity fields. The performance of these generative models is evaluated in terms of the distribution of the mean drag coefficient and the root mean squared control-point velocity field.

Since the simulations are carried out with OpenFOAM, the numerical solver is fully Eulerian and therefore requires the reconstruction of the entire mesh trajectory. This procedure is computationally more expensive and technically more demanding, as mesh orthogonality may be degraded during the deformation process.

Because solving the full PDE system for each velocity field is computationally intensive, several ROMs—namely GPR, KNN, RF-are investigated. These ROMs are trained both on the original control-point velocity fields and on the latent representations of the velocity fields generated by the generative models. In addition to assessing ROM accuracy via LOOCV error, we also compare the optimisation performance of the different ROMs.

For the FFD–ODE-based ROM, we consider the following constrained optimisation problem:
\[
\begin{array}{rcll}
\displaystyle \min\limits_{\boldsymbol{\alpha}\in C([0,1],\mathbb{R}^{5\times5\times5})} &~& \mathcal{G}_{\mathrm{FFD\text{-}ODE}}(\boldsymbol{\alpha})  & \\
\mathrm{subject~to} &~& \displaystyle \frac{1}{102}\sum_{i=0}^{101}\left\|\boldsymbol{\alpha}\!\left(\frac{i}{101}\right)\right\|_{2}\le 1.65, \nonumber  \\
 &~& \left\|\boldsymbol{\alpha}(i)-\boldsymbol{\alpha}(i-1)\right\|_{2}\le 0.01, 
 &\forall i=1,\ldots,101 \\
\end{array}
\label{eq:ffd_ode_opt}
\]
where the constraints are imposed to prevent non-physical deformations and excessive mesh distortion.

For the generative-model-based approach, we solve the unconstrained optimisation problem
\[
\min_{\boldsymbol{z}\in \mathbb{R}^{10}}\mathcal{G}_{\mathrm{GEN}}(\boldsymbol{z}),
\]
since the STARFlow-generated samples already satisfy the physical and geometric constraints with high probability, owing to the implicit prior learned from the training data, which itself respects these constraints.

The functionals $\mathcal{G}_{\mathrm{FFD\text{-}ODE}}$ and $\mathcal{G}_{\mathrm{GEN}}$ denote the GPR surrogate (which exhibits the lowest prediction error among the considered ROMs) used to approximate the drag coefficient. The optimisation problems are solved using the \texttt{trust-constr} algorithm~\cite{byrd_interior_1999}.

\begin{figure}[ht]
    \centering
    \includegraphics[width=1\linewidth]{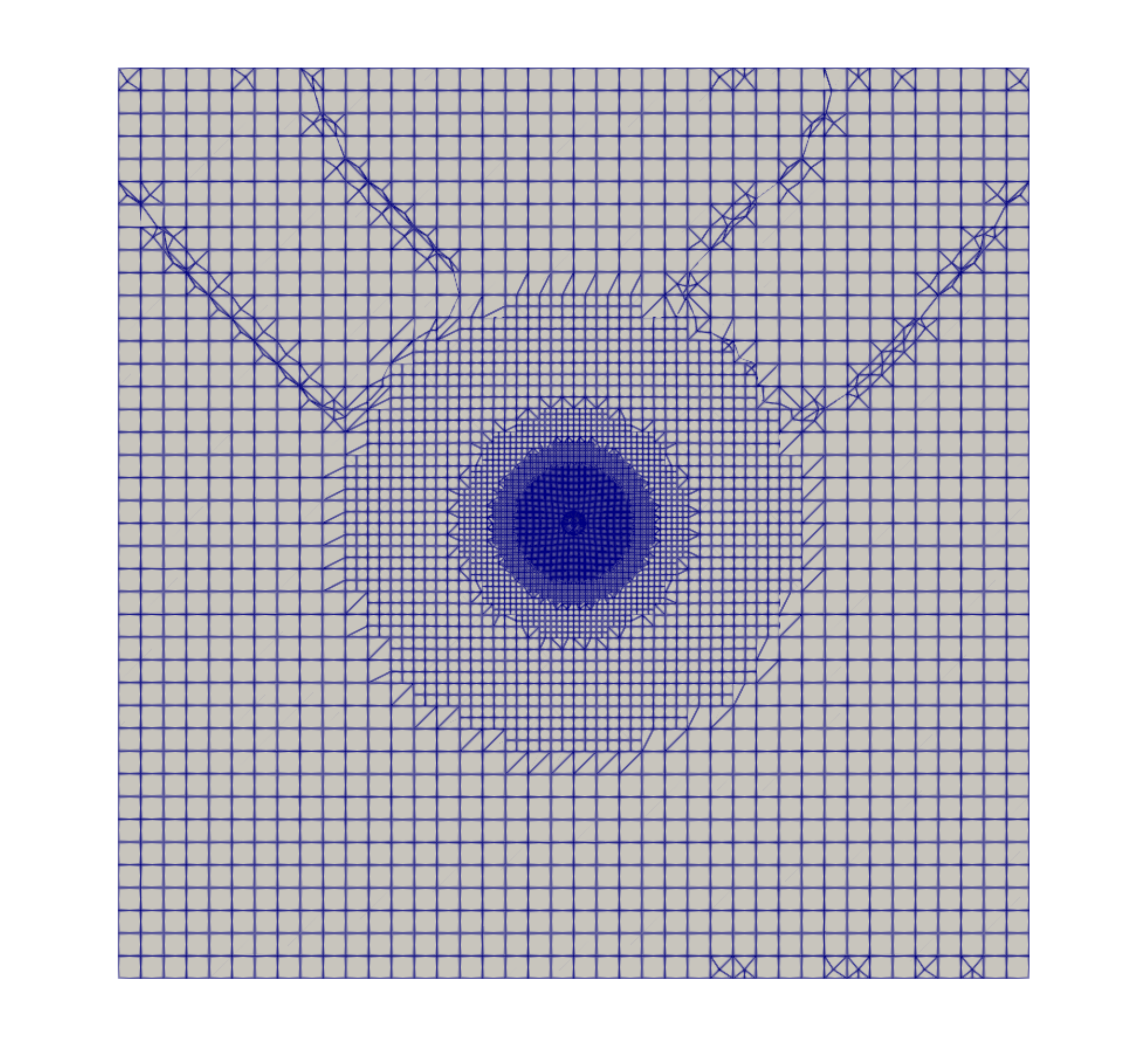}
    \caption{Detail of the computational mesh adopted for the sphere test case.}
    \label{fig:3}
\end{figure}

\section{Numerical results}
\label{sec:results}

In this section, we present results demonstrating the application of our methodology to an incompressible flow over a moving Stanford Bunny and to a compressible flow over a deforming spherical geometry.

\subsection{Stanford Bunny test case}
In Fig. \ref{fig:4}, the target configurations obtained with the FFD-ODE framework and with the generative models are reported, together with the corresponding Navier–Stokes solutions. A high degree of variability in the flow fields is shown.

\begin{figure}[ht]
    \centering
    \includegraphics[width=0.3\linewidth]{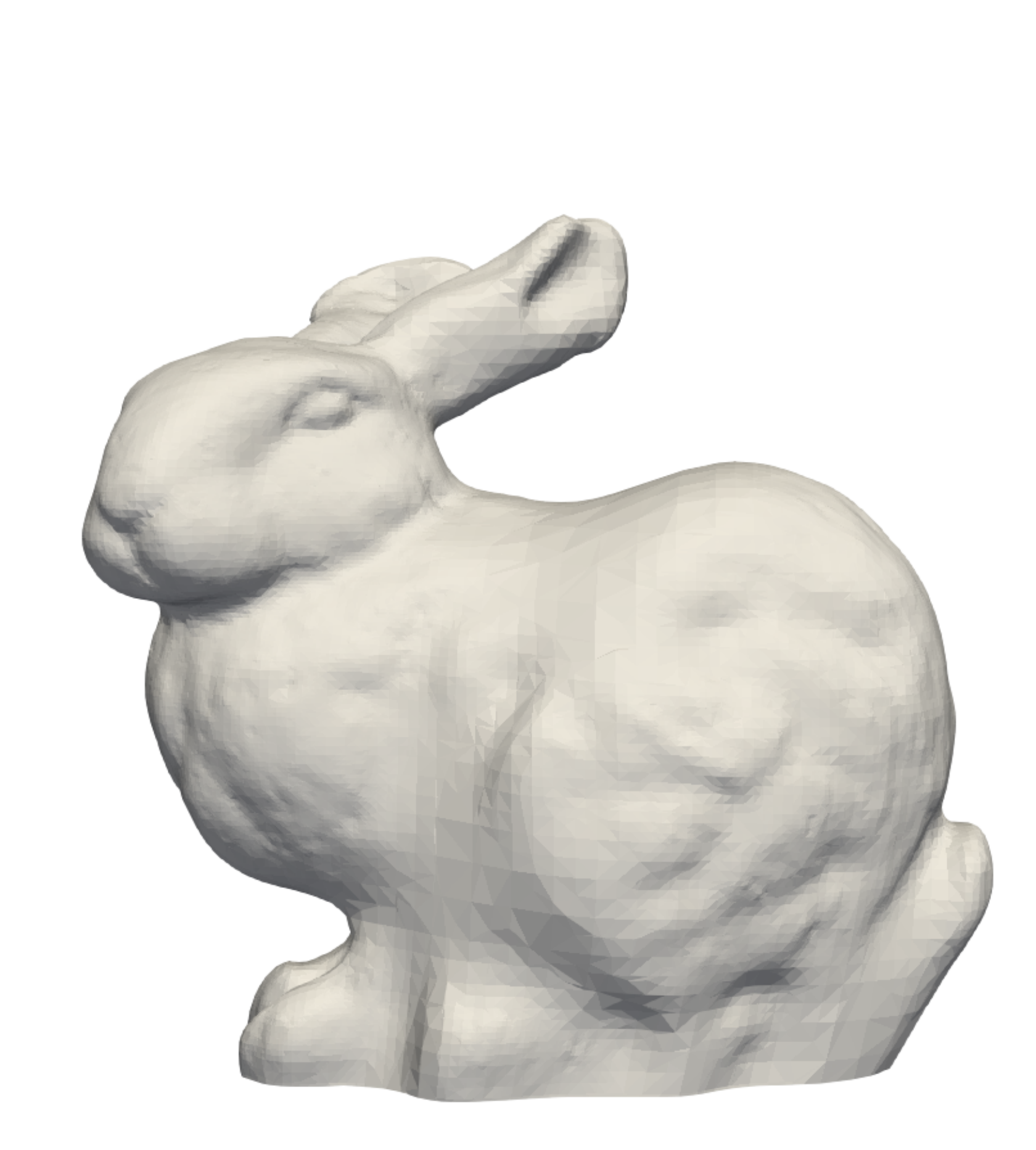}
\includegraphics[width=0.3\linewidth]{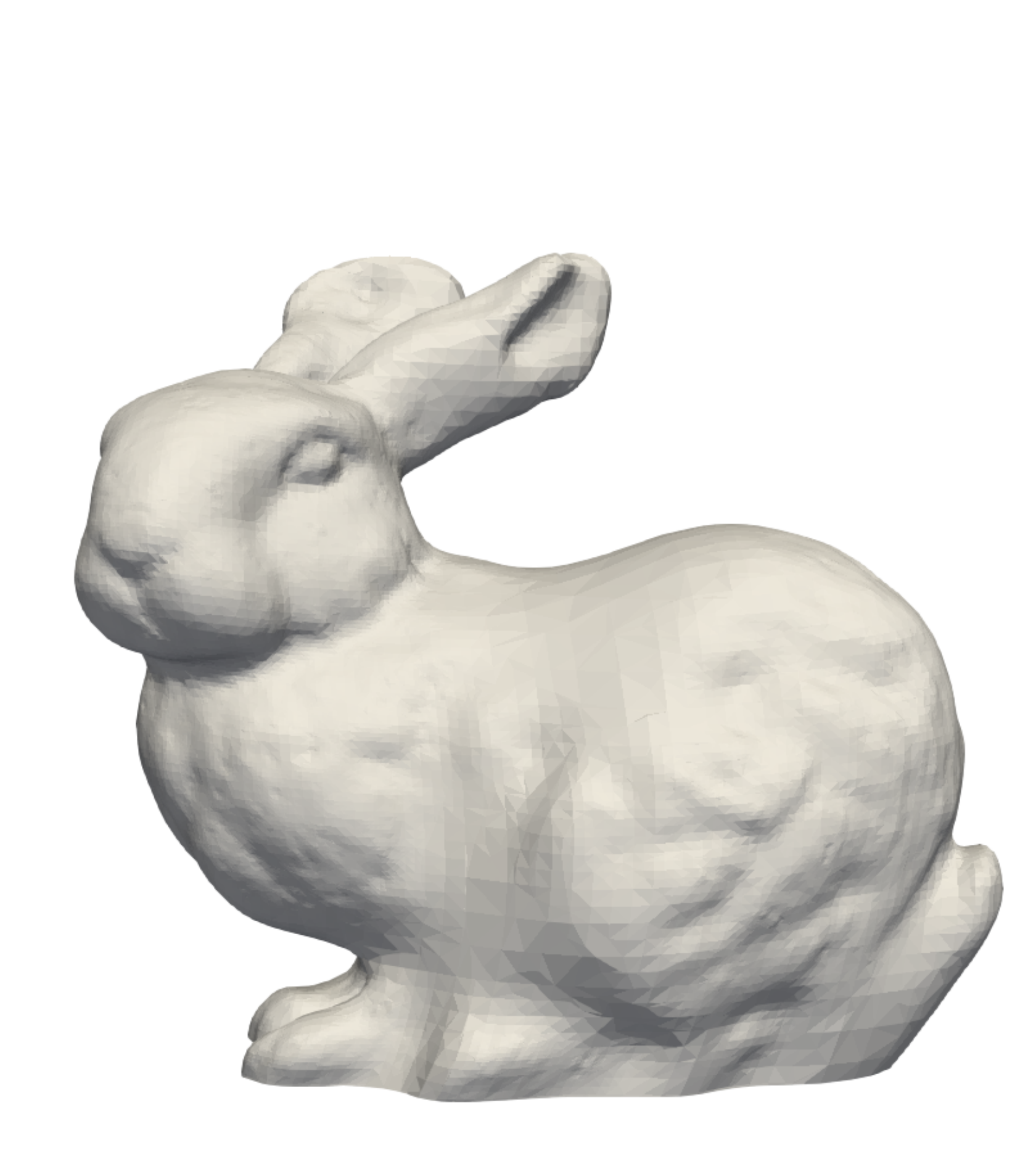}
\includegraphics[width=0.3\linewidth]{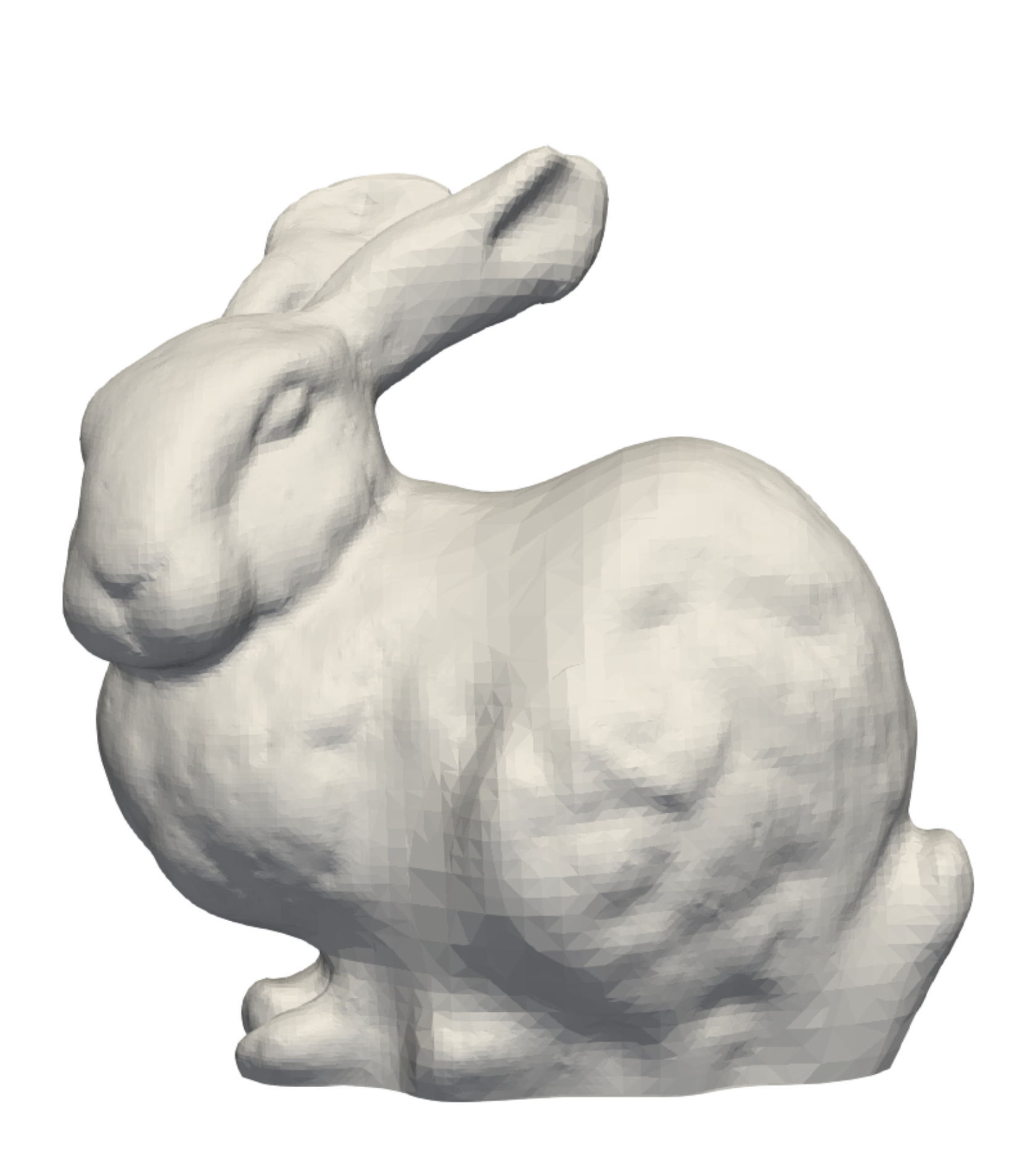}\\
    \includegraphics[width=0.3\linewidth]{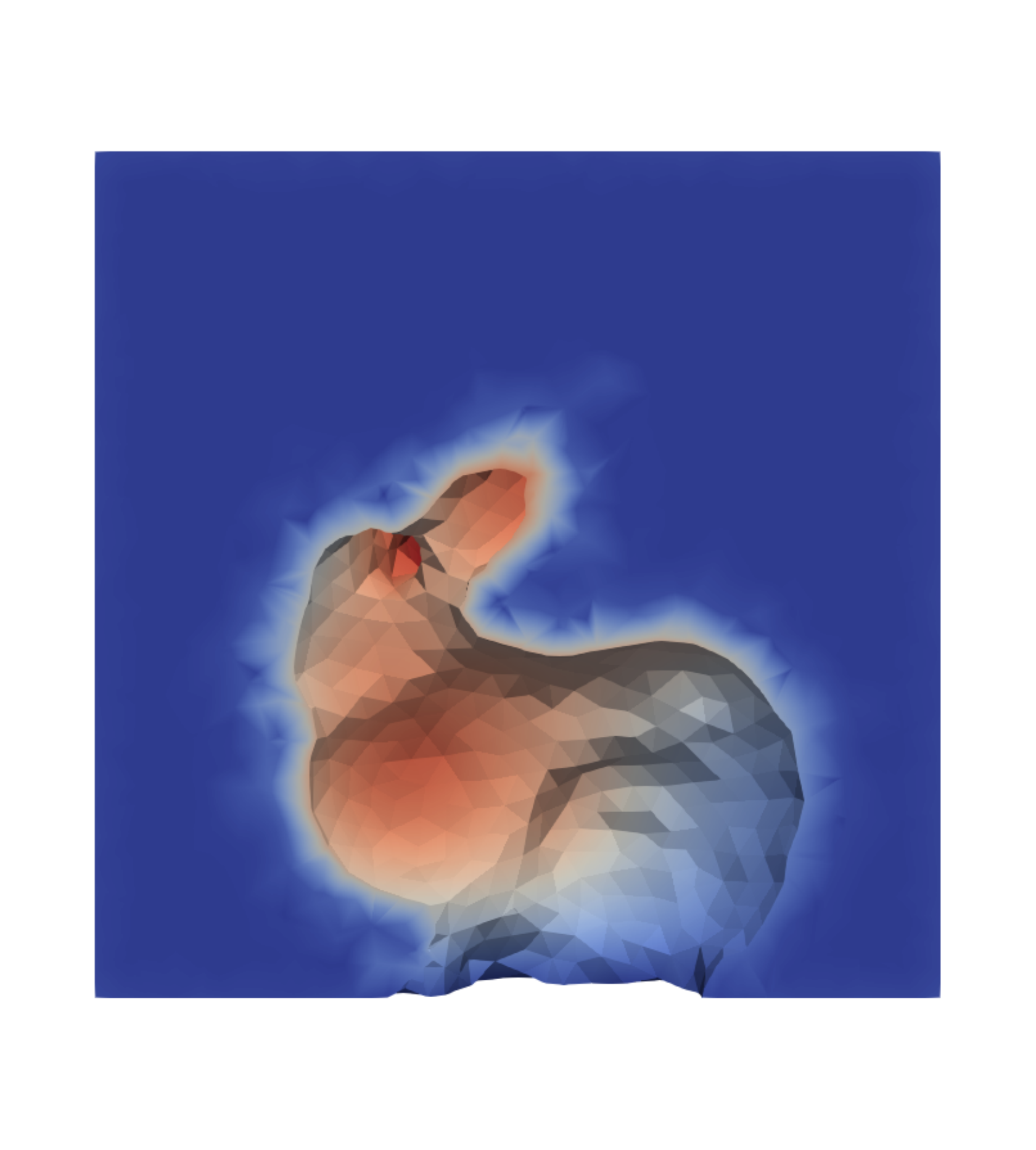}
\includegraphics[width=0.3\linewidth]{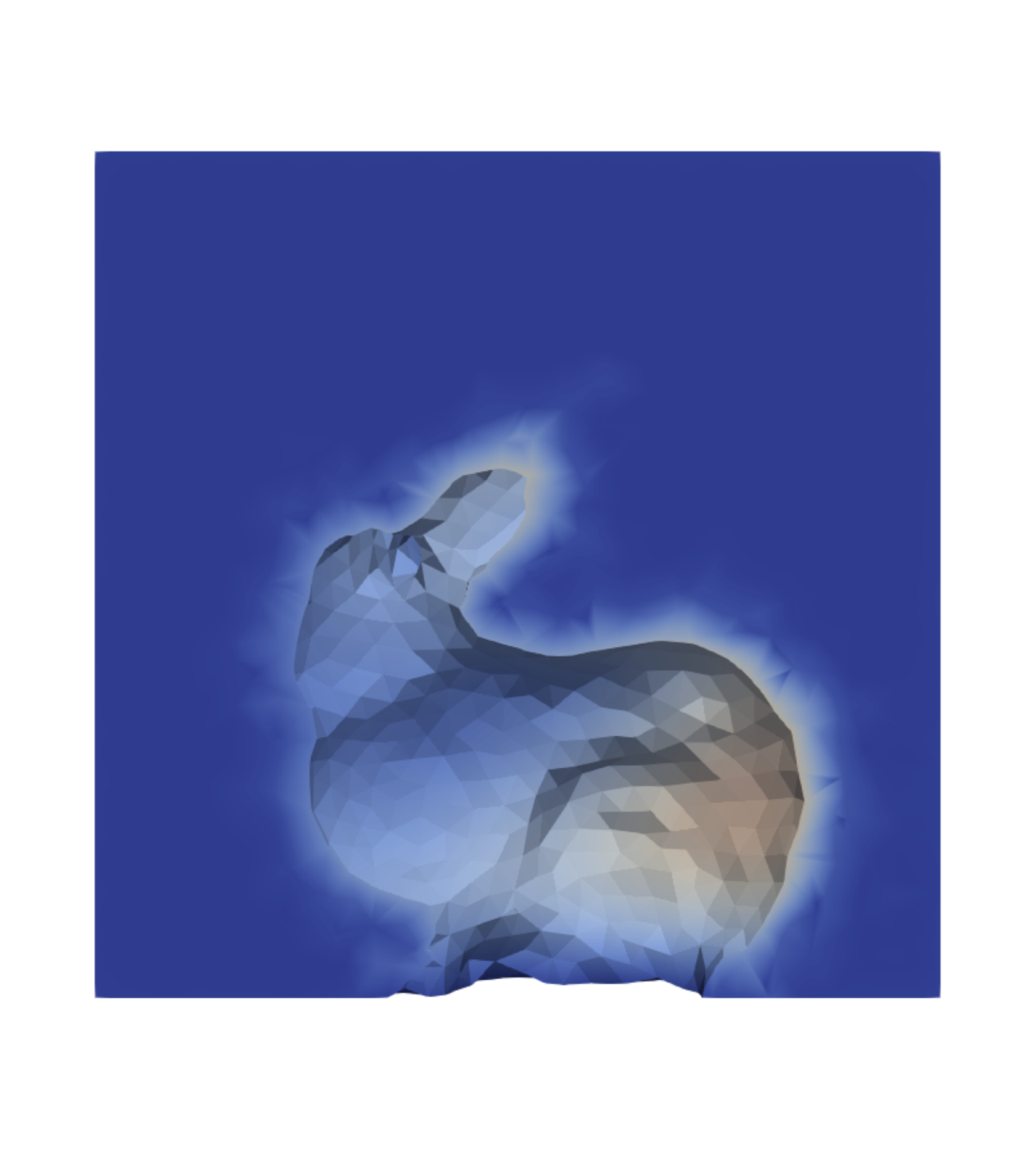}
\includegraphics[width=0.3\linewidth]{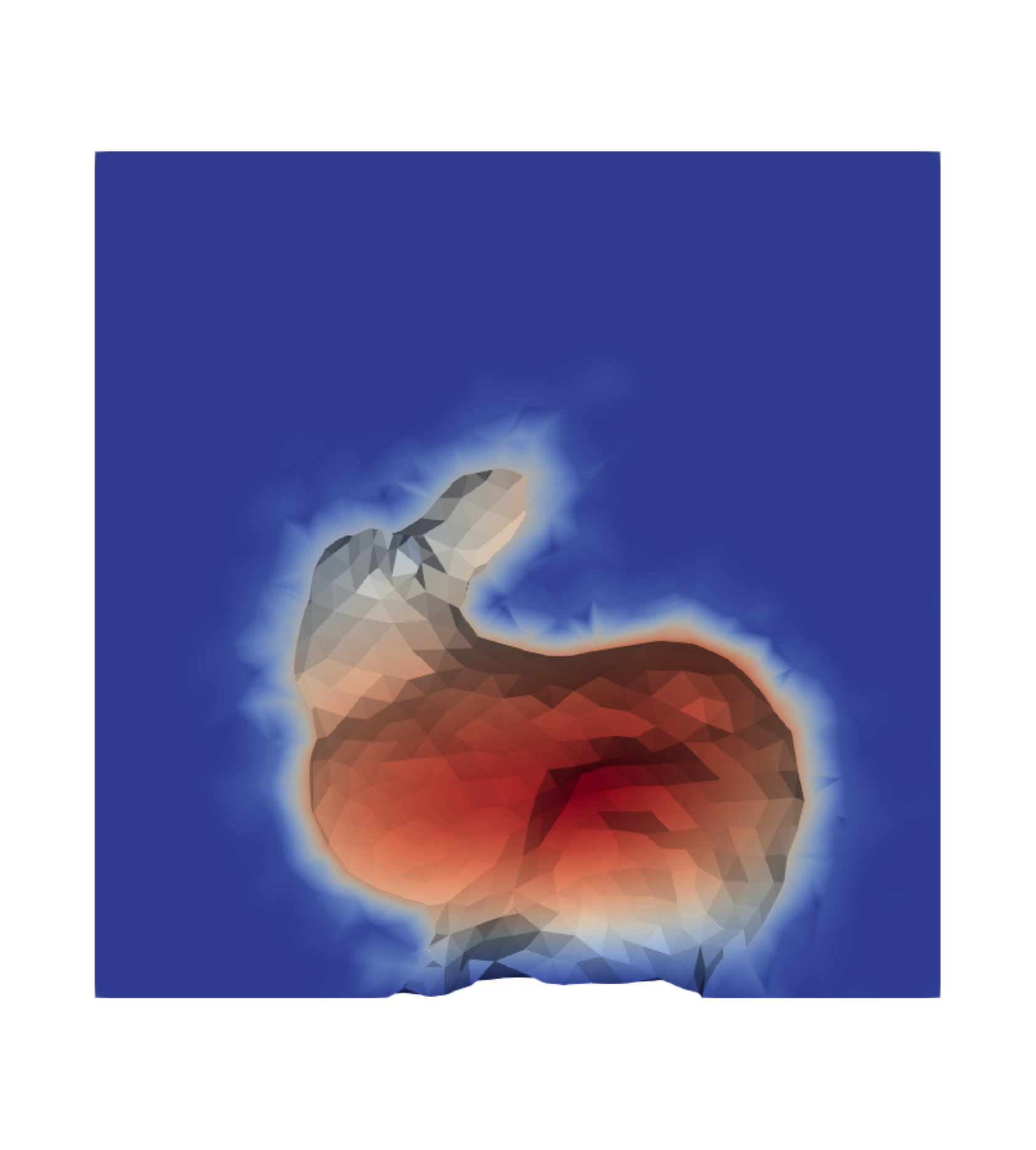}\\
\includegraphics[width=1\linewidth]{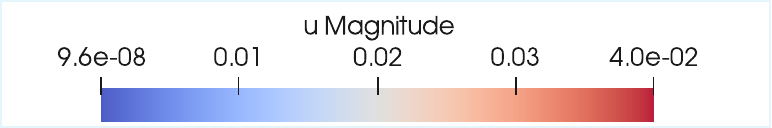}

\caption{First row: bunny deformation at time $t=1$ obtained with FFD-ODE (first two panels) and with the generative model (last panel). Second row: corresponding Navier–Stokes velocity fields. Marked discrepancies among the solutions can be observed.}
\label{fig:4}
\end{figure}

In Figs. \ref{fig:5} and \ref{fig:6} we show the empirical distributions of the mean flow velocity, defined as
\[
\frac{1}{\Omega_{0}}\int_{\Omega_{0}}\int_{0}^{1}\|\boldsymbol{u}_{h}(\boldsymbol{x},t)\|^{2}\,\mathrm{d}t\,\mathrm{d}\boldsymbol{x},
\]
and of the time-averaged squared norm of the control-point velocity field,
defined as 
\[
\int_{0}^{1}\|\boldsymbol{\alpha}(t)\|^{2}\,\mathrm{d}t,
\]
for both the FFD-ODE-based and the generative-model-based outputs. The distributions obtained with STARFlow closely match those associated with the FFD-ODE model.

\begin{figure}[ht]
    \centering
    \includegraphics[width=1\linewidth]{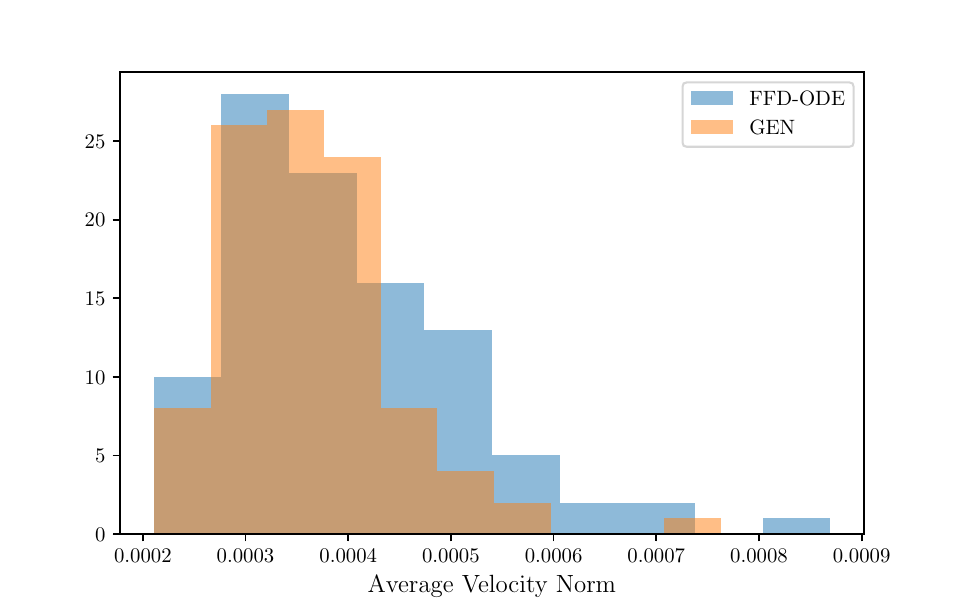}

\caption{Distribution of the average velocity norm for the FFD-ODE outputs and the STARFlow-generated outputs. STARFlow accurately reproduces the reference distribution.}
\label{fig:5}
\end{figure}

\begin{figure}[ht]
    \centering
\includegraphics[width=1\linewidth]{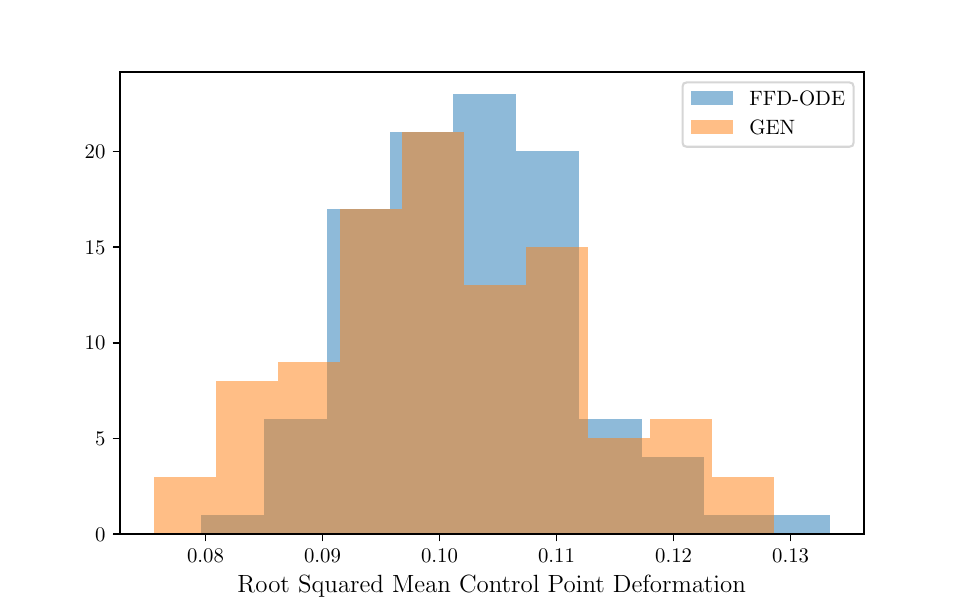}

\caption{Distribution of the root mean square control-point deformation for the FFD-ODE outputs and the STARFlow-generated outputs. STARFlow accurately reproduces the reference distribution.}
\label{fig:6}
\end{figure}

Tables \ref{l1:bunny} and \ref{l2:bunny} report the LOOCV L2 and L1 errors, respectively, of KNN, GPR, and RF models trained on the normalised FFD-ODE average velocity norms or on the normalised STARFlow average velocity norms. Surrogate models trained on the STARFlow-based quantities exhibit systematically lower errors, which we attribute to the reduced parametrisation afforded by the generative model.

\begin{table}[H]
\centering
\begin{tabular}{lcc}
\hline
 & \textbf{FFD-ODE} & \textbf{GEN} \\
\hline
KNN & 0.1906 & 0.1512 \\
GPR & 0.1672   & 0.1448 \\
RF  & 0.1723 & 0.1321 \\\hline
\hline
\end{tabular}
\caption{LOOCV L2 error of KNN, GPR, and RF when applied to normalised FFD-ODE average velocity norms and to normalised STARFlow average velocity norms. Models trained on STARFlow-based quantities yield lower errors, reflecting the advantages of the reduced parametrisation.}
\label{l1:bunny}

\end{table}

\begin{table}[H]
\centering
\begin{tabular}{lcc}
\hline
 & \textbf{FFD-ODE} & \textbf{GEN} \\
\hline
KNN & 0.1430 & 0.1090 \\
GPR & 0.1423 & 0.1080 \\
RF  & 0.1354 & 0.0979 \\\hline
\hline
\end{tabular}
\caption{LOOCV L1 error of KNN, GPR, and RF when applied to normalised FFD-ODE average velocity norms and to normalised STARFlow average velocity norms. The STARFlow-based parametrisation leads to improved predictive performance for all surrogate models.}
\label{l2:bunny}
\end{table}

\subsection{Flow past a sphere test case}

In Fig. \ref{fig:7}, representative temperature fields obtained from the generative models are displayed, illustrating the variability of the predicted solutions. In Figs. \ref{fig:8} and \ref{fig:9}, we present the distributions of the drag coefficient and of the time-averaged squared norm of the control-point velocity field for both the FFD-ODE model and the generative models. The STARFlow-based distributions are in close agreement with those of the FFD-ODE reference model.

In Tables \ref{l1:sphere} and \ref{l2:sphere}, we report the LOOCV L2 and L1 errors of KNN, GPR, and RF when trained on normalised FFD-ODE drag values and on normalised STARFlow drag values. Surrogate models trained on STARFlow-based drags consistently achieve lower errors, again reflecting the beneficial effect of the reduced parametrisation.

The optimisation performance is summarised in Fig. \ref{fig:10}. ROMs based on STARFlow exhibit superior optimisation behaviour owing to the reduced latent-space dimension and the implicit regularisation induced by the generative model. For the STARFlow-based ROMs, the drag-minimisation problem converges in 394 iterations, requiring approximately 5 seconds of wall-clock time, and the continuity constraint is satisfied despite not being explicitly enforced in the optimiser, a consequence of the generative prior. In contrast, the FFD-ODE-based ROMs become trapped in a local minimum: the \texttt{trust-constr} algorithm reaches its maximum of 1000 iterations without convergence, in about 3 hours, although the constraint is satisfied. This corresponds to at least a $3600\times$ speed-up when employing the generative models. All simulations and optimisation runs are performed on an HP laptop equipped with 16 GB RAM and an Intel i5-1340P processor.

\begin{table}[H]
\centering
\begin{tabular}{lcc}
\hline
 & \textbf{FFD-ODE} & \textbf{GEN} \\
\hline
KNN & 0.2246 & 0.1845 \\
GPR & 0.1620   &  0.1540 \\
RF  & 0.2060 &  0.1863 \\\hline
\end{tabular}
\caption{LOOCV L2 error of KNN, GPR, and RF when applied to normalised FFD-ODE drags and to normalised STARFlow drags. The STARFlow-based parametrisation yields lower errors, highlighting the benefit of the reduced latent representation.}
\label{l1:sphere}
\end{table}

\begin{table}[H]
\centering
\begin{tabular}{lcc}
\hline
 & \textbf{FFD-ODE} & \textbf{GEN} \\
\hline
KNN & 0.1830 & 0.1464 \\
GPR & 0.1510 & 0.1168 \\
RF  & 0.1669 & 0.1447 \\
\hline
\end{tabular}
\caption{LOOCV L1 error of KNN, GPR, and RF when applied to normalised FFD-ODE drags and to normalised STARFlow drags. Models trained on STARFlow outputs exhibit improved accuracy due to the reduced parameterisation.}
\label{l2:sphere}
\end{table}

\begin{figure}
    \centering
    \includegraphics[width=0.3\linewidth]{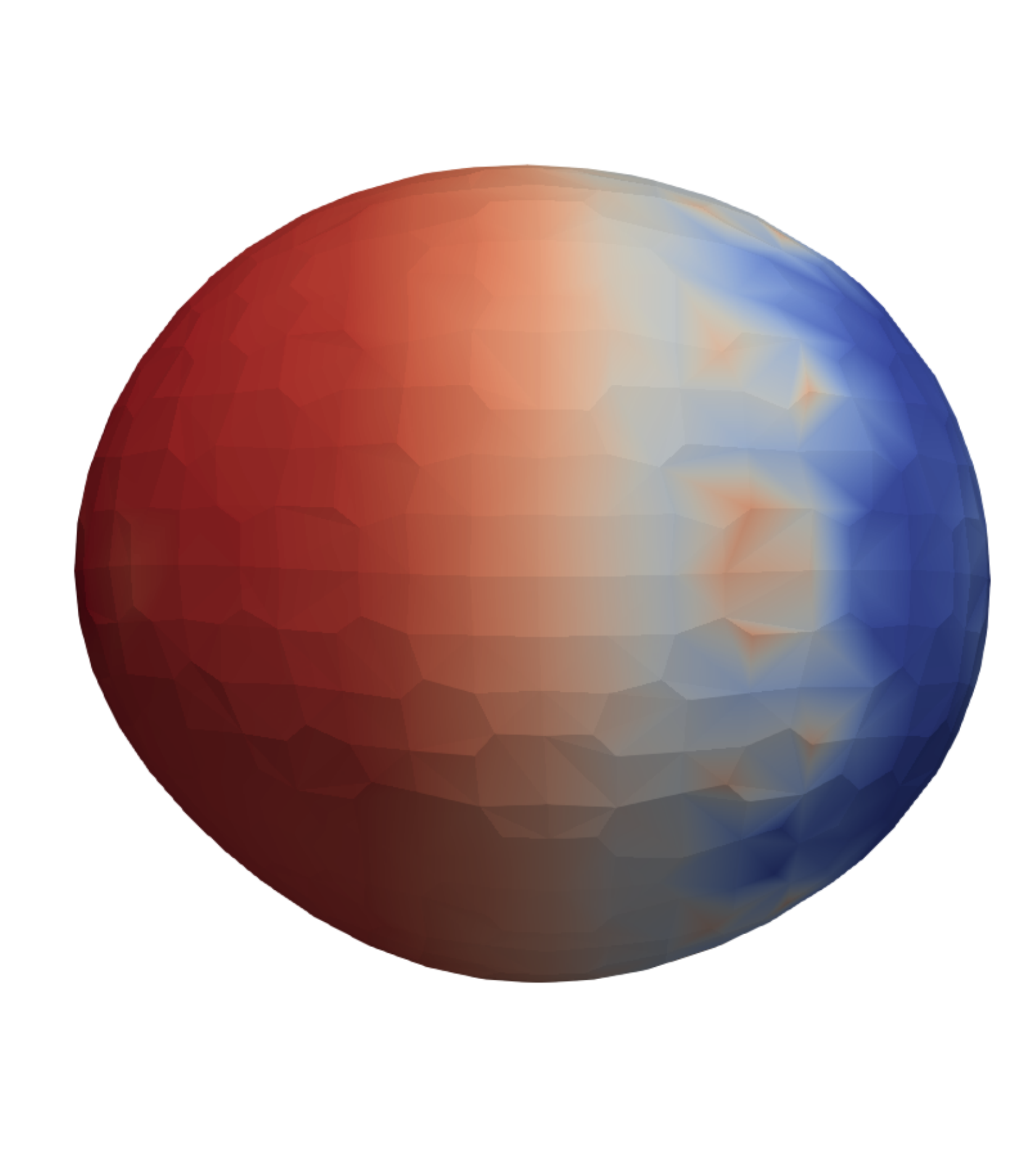}
\includegraphics[width=0.3\linewidth]{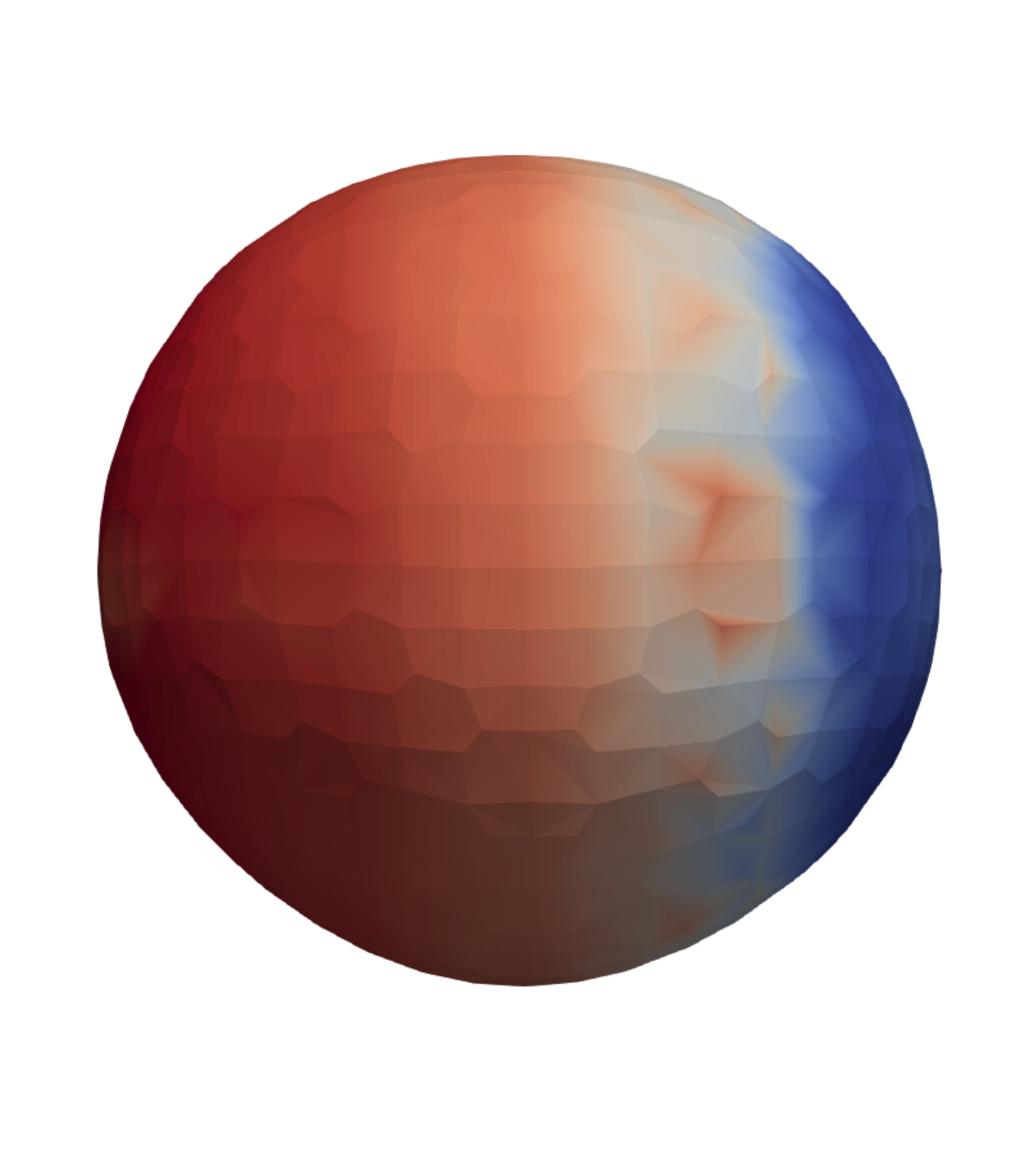}
\includegraphics[width=0.3\linewidth]{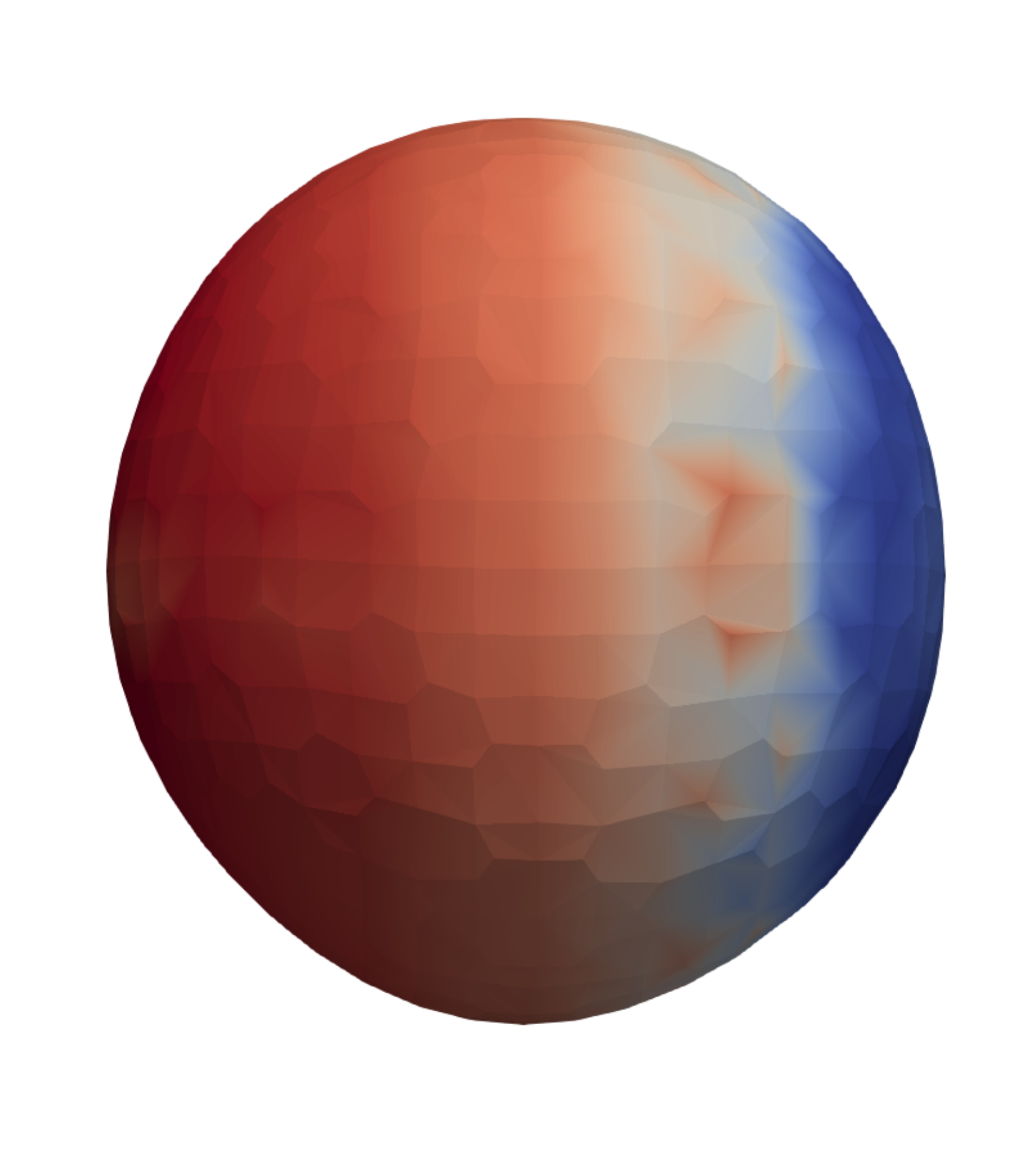}\\
\includegraphics[width=1\linewidth]{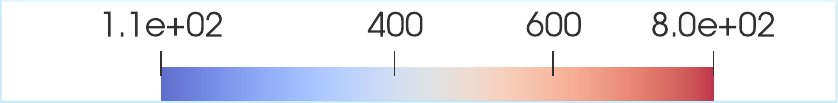}
\caption{First row: temperature field at time $t=1$ from FFD-ODE (first two panels) and from the generative model (last panel). Pronounced variability in the solutions can be observed.}
\label{fig:7}
\end{figure}
\begin{figure}
    \centering
    \includegraphics[width=1\linewidth]{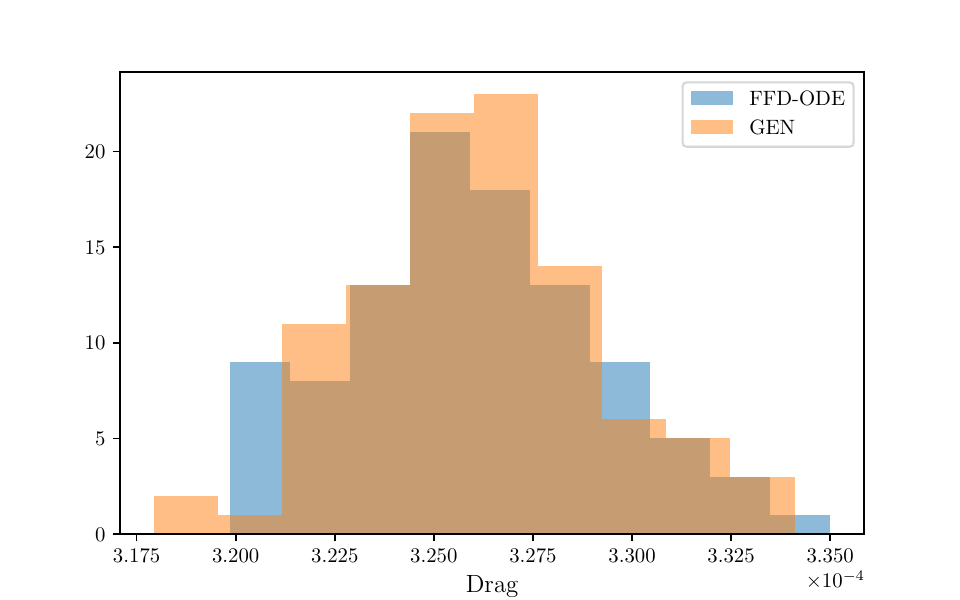}

\caption{Distribution of drag for the FFD-ODE outputs and the STARFlow-generated outputs. STARFlow accurately reproduces the drag distribution.}
\label{fig:8}
\end{figure}

\begin{figure}
    \centering
\includegraphics[width=1\linewidth]{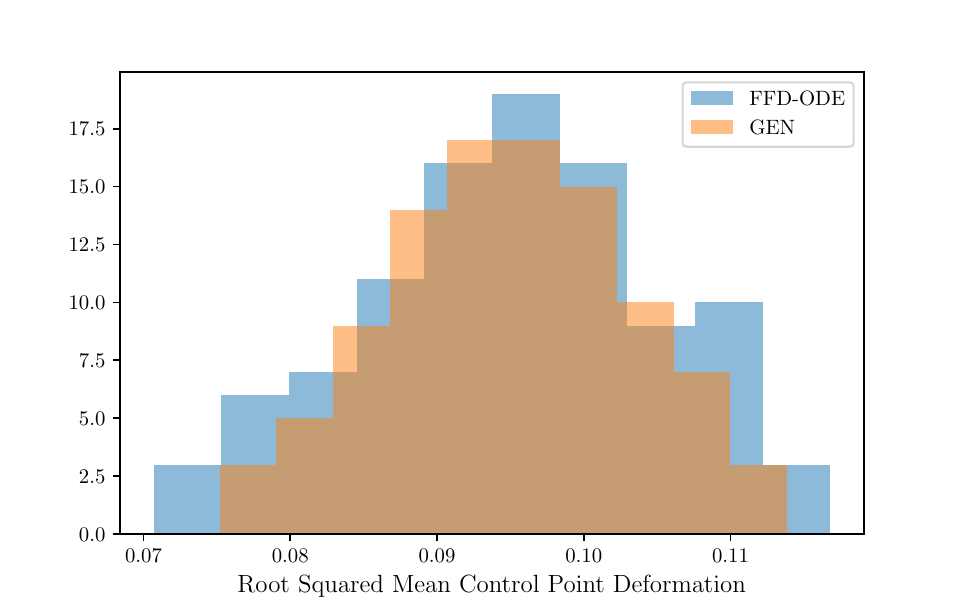}

\caption{Distribution of the root mean square control-point deformation for the FFD-ODE outputs and the STARFlow-generated outputs. STARFlow accurately reproduces the reference distribution.}
\label{fig:9}
\end{figure}

\begin{figure}
    \centering
\includegraphics[width=1\linewidth]{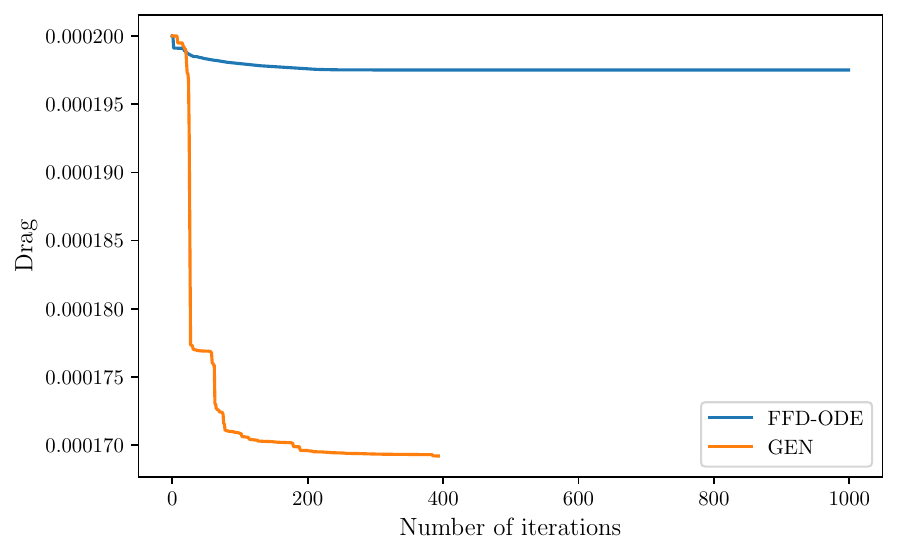}

\caption{Optimisation performance for drag minimisation with FFD-ODE-based ROMs and STARFlow-based ROMs. The STARFlow-based ROMs show superior performance owing to the reduced parametrisation and the implicit regularisation provided by the generative model.}
\label{fig:10}
\end{figure}

\section{Conclusion}
\label{sec:conclusion}

We introduce a novel FFD-ODE framework for time-dependent mesh deformation that embeds Free Form Deformation mappings within an ODE drift formulation. We establish theoretical results demonstrating universal approximation properties on genus-0 surfaces, and we construct a generative model based on POD and STARFlow that enables an efficient, low-dimensional parametrisation of control velocity fields. 

The approach is validated on two benchmark problems: incompressible flow past a deforming Stanford Bunny, discretised in FEniCSx with Streamline-Upwind Petrov–Galerkin (SUPG) stabilisation, and hypersonic flow past a deforming sphere, simulated with OpenFOAM. In both cases, the generative model accurately reproduces the distributions of QoIs, and ROMs constructed on the learned latent space yield a 30–40\% reduction in error relative to classical high-dimensional parametrisations.

The proposed methodology integrates classical geometric morphing techniques with modern data-driven approaches, providing a topologically robust and interpretable framework for shape registration and reduced-order modelling. Future research directions include extension of the framework to higher-genus surfaces, incorporation of physics-informed constraints within the generative model, and exploration of real-time optimisation applications in aerodynamic design and biomedical imaging.

\section*{Acknowledgements}
The authors acknowledge the support provided by INdAM-GNCS and the European Union NextGenerationEU, in the framework of the iNEST- Interconnected Nord-Est Innovation Ecosystem (iNEST ECS00000043 - CUP G93C22000610007) consortium.

\clearpage
\bibliographystyle{unsrt} 
\bibliography{export}
\clearpage
\appendix
\section{Proof of Theorem 2}

\begin{proof}
By \cite{chazal_condition_2005}, there exists a diffeomorphism from \(S_{1}\) to \(S_{2}\) that is isotopic to the identity. By \cite{hirsch_differential_1976}, Theorem 1.3 on p.~180, this isotopy can be extended to a smooth diffeotopy \(\boldsymbol{h} : \mathbb{R}^{3} \times [0,1] \rightarrow \mathbb{R}^{3}\), supported in \((0,1)^{3}\), such that \(\boldsymbol{h}(S_{1},1) = S_{2}\) and \(\boldsymbol{h}(S_{1},0) = S_{1}\).

For each fixed \(t\), let \(\boldsymbol{h}^{-1}\) denote the inverse of \(\boldsymbol{h}(\cdot,t)\). Define
\[
\boldsymbol{x}(t) = \boldsymbol{h}(\boldsymbol{x}_{0},t).
\]
Then, it holds that
\[
\dot{\boldsymbol{x}}(t) = \frac{\partial \boldsymbol{h}}{\partial t}(\boldsymbol{x}_{0},t) = \frac{\partial \boldsymbol{h}}{\partial t}\big(\boldsymbol{h}^{-1}(\boldsymbol{x}(t),t),t\big).
\]
Let us set
\[
\boldsymbol{g}(\boldsymbol{x},t) = \frac{\partial \boldsymbol{h}}{\partial t}\big(\boldsymbol{h}^{-1}(\boldsymbol{x},t),t\big).
\]
With this notation, the isotopy can be realised as the flow of the time-dependent vector field \(\boldsymbol{g}\), namely as the solution of
\[
\begin{cases}
\dot{\boldsymbol{x}}(t) = \boldsymbol{g}(\boldsymbol{x}(t),t),\\
\boldsymbol{x}(0) = \boldsymbol{x}_{0}.
\end{cases}
\]
Since \(\boldsymbol{h}\) is a smooth diffeotopy with support included in a compact set, the vector field \(\boldsymbol{g}\) is smooth with all derivatives bounded on $\mathbb{R}^{3}$. Thus, it is  Lipschitz continuous in \(\boldsymbol{x}\) with a bounded Lipschitz constant \(L\). By Theorem \ref{thm:1}, for any \(\nu > 0\) there exist integers \(m,n,o\) sufficiently large such that
\[
\sup_{t \in [0,1],\, \boldsymbol{x} \in [0,1]^{3}} \big\|\boldsymbol{F}_{\boldsymbol{\alpha}(t)}(\boldsymbol{x}) - \boldsymbol{g}(\boldsymbol{x},t)\big\| \le \nu.
\]

Let \(\boldsymbol{y}(t)\) denote the solution of
\[
\begin{cases}
\dot{\boldsymbol{y}}(t) = \boldsymbol{F}_{\boldsymbol{\alpha}(t)}(\boldsymbol{y}(t)),\\
\boldsymbol{y}(0) = \boldsymbol{x}_{0},
\end{cases}
\]
and define the error
\[
\boldsymbol{e}(t) = \boldsymbol{y}(t) - \boldsymbol{x}(t).
\]
Then
\[
\begin{gathered}
\dot{\boldsymbol{e}}(t) =\boldsymbol{F}_{\boldsymbol{\alpha}(t)}(\boldsymbol{y}(t)) - \boldsymbol{g}(\boldsymbol{x}(t),t)  \\
= \big[\boldsymbol{F}_{\boldsymbol{\alpha}(t)}(\boldsymbol{y}(t)) - \boldsymbol{g}(\boldsymbol{y}(t),t)\big] + \big[\boldsymbol{g}(\boldsymbol{y}(t),t) - \boldsymbol{g}(\boldsymbol{x}(t),t)\big].
\end{gathered}
\]
In integral form, this yields
\[
\boldsymbol{e}(t) = \int_{0}^{t} \big(\boldsymbol{F}_{\boldsymbol{\alpha}(s)}(\boldsymbol{y}(s)) - \boldsymbol{g}(\boldsymbol{y}(s),s)\big)\, ds \,+\, \int_{0}^{t} \big(\boldsymbol{g}(\boldsymbol{y}(s),s) - \boldsymbol{g}(\boldsymbol{x}(s),s)\big)\, ds.
\]
Taking norms and using the approximation property of \(\boldsymbol{F}_{\boldsymbol{\alpha}(s)}\) together with the Lipschitz continuity of \(\boldsymbol{g}\), we obtain
\[
\|\boldsymbol{e}(t)\| \le \nu t + L \int_{0}^{t} \|\boldsymbol{e}(s)\|\, ds.
\]
An application of Grönwall’s inequality then gives
\[
\|\boldsymbol{e}(t)\| \le \frac{\nu}{L}\big(\exp(L t) - 1\big).
\]
By choosing \(\nu = \dfrac{L}{\exp(L) - 1}\,\epsilon\), we deduce that
\[
\|\boldsymbol{e}(t)\| \le \epsilon \quad \forall\, t \in [0,1],\, \boldsymbol{x}_{0} \in [0,1]^{3}.
\]
Evaluating at \(t=1\), we obtain
\[
\|\boldsymbol{G}_{\boldsymbol{\alpha}}(\boldsymbol{x}) - \boldsymbol{h}(\boldsymbol{x},1)\| \le \epsilon \quad \forall\, \boldsymbol{x} \in S_1.
\]
Since \(\boldsymbol{h}(\cdot,1)\) is a bijection from \(S_1\) onto \(S_2\), it follows that for every \(\boldsymbol{x} \in S_1\),
\[
\operatorname{dist}\big(\boldsymbol{G}_{\boldsymbol{\alpha}}(\boldsymbol{x}), S_2\big)
\le \|\boldsymbol{G}_{\boldsymbol{\alpha}}(\boldsymbol{x}) - \boldsymbol{h}(\boldsymbol{x},1)\|
\le \epsilon.
\]
Conversely, for every \(\boldsymbol{y} \in S_2\), there exists \(\boldsymbol{x} \in S_1\) such that \(\boldsymbol{y} = \boldsymbol{h}(\boldsymbol{x},1)\), and hence
\[
\operatorname{dist}\big(\boldsymbol{y}, \boldsymbol{G}_{\boldsymbol{\alpha}}(S_1)\big)
\le \|\boldsymbol{y} - \boldsymbol{G}_{\boldsymbol{\alpha}}(\boldsymbol{x})\|
\le \epsilon.
\]
Therefore,
\[
CD\big(\boldsymbol{G}_{\boldsymbol{\alpha}}(S_1), S_2\big) \le \epsilon,
\]
which proves the claim.
Finally, since \(\boldsymbol{g}\) is supported in \((0,1)^{3}\), the coefficients associated with boundary grid points are necessarily zero. This yields the stated constraint on the weights \(\alpha_{ijk}\).
\end{proof}

\section{Proof of Theorem 3}
We need the following propositions.
\begin{proposition}
\label{lem:basis}
Let $m,n,o\ge 2$ and $\bar m=\max\{m,n,o\}$. For all
$\boldsymbol{x},\boldsymbol{y}\in[0,1]^{3}$:
\begin{enumerate}
\item[(a)] $\|\boldsymbol{B}(\boldsymbol{x})\|_{1}=1$;
\item[(b)] $\|\boldsymbol{B}(\boldsymbol{x})-\boldsymbol{B}(\boldsymbol{y})\|_{1}
      \le 2\bar m\,\|\boldsymbol{x}-\boldsymbol{y}\|_{1}$;
\item[(c)] $\displaystyle\sum_{c\in\mathcal{I}^{\circ}}
      B_{c}(\boldsymbol{x})^{2}\le \tfrac18$, and consequently, for every
      $\boldsymbol{v}\in\mathcal{C}_{0}$,
      \[
      \|\boldsymbol{v}\,\boldsymbol{B}(\boldsymbol{x})\|_{1}
      \le \sqrt{3}\,\|\boldsymbol{v}\,\boldsymbol{B}(\boldsymbol{x})\|_{2}
      \le \sqrt{\tfrac{3}{8}}\;\|\boldsymbol{v}\|_{\mathcal{C}} ;
      \]
\item[(d)] for every $\boldsymbol{v}\in\mathcal{C}$ and
      $\boldsymbol{w}\in\mathbb{R}^{(m+1)(n+1)(o+1)}$,
      $\|\boldsymbol{v}\,\boldsymbol{w}\|_{1}
      \le \sqrt{3}\,\|\boldsymbol{v}\|_{\mathcal{C}}\,\|\boldsymbol{w}\|_{1}$.
\end{enumerate}
\end{proposition}

\begin{proof}
(a) On $[0,1]$ the $b_{m,i}$ are nonnegative and sum to one; the
trivariate basis inherits both properties by taking products, hence
$\|\boldsymbol{B}(\boldsymbol{x})\|_{1}
=\sum_{c}B_{c}(\boldsymbol{x})=1$.

(b) From $b'_{m,i}=m\,(b_{m-1,i-1}-b_{m-1,i})$ (with
$b_{m-1,-1}=b_{m-1,m}\equiv 0$), nonnegativity and (a) give
$\sum_{i}|b'_{m,i}(x)|\le m\sum_{i}\bigl(b_{m-1,i-1}(x)+b_{m-1,i}(x)\bigr)=2m$
for $x\in[0,1]$. Telescoping over the three coordinates and using (a)
in the frozen variables,
\[
\|\boldsymbol{B}(\boldsymbol{x})-\boldsymbol{B}(\boldsymbol{y})\|_{1}
\le 2m|x_{1}-y_{1}|+2n|x_{2}-y_{2}|+2o|x_{3}-y_{3}|
\le 2\bar m\,\|\boldsymbol{x}-\boldsymbol{y}\|_{1}.
\]

(c) We first claim $b_{m,i}(x)\le\frac12$ on $[0,1]$ whenever
$0<i<m$. Indeed, for such $i$,
\[
b_{m,i}(x)^{2}
= r\; b_{m,i-1}(x)\,b_{m,i+1}(x),
\qquad
r:=\frac{\binom{m}{i}^{2}}{\binom{m}{i-1}\binom{m}{i+1}}
 =\frac{(i+1)(m-i+1)}{i\,(m-i)}\le 4,
\]
where the bound on $r$ follows from $(i+1)/i\le 2$ and
$(m-i+1)/(m-i)\le 2$. By the AM--GM inequality and (a),
$b_{m,i-1}b_{m,i+1}\le\frac14(b_{m,i-1}+b_{m,i+1})^{2}
\le\frac14(1-b_{m,i})^{2}$, hence
$b_{m,i}\le\frac{\sqrt r}{2}(1-b_{m,i})$ and
$b_{m,i}\le\frac{\sqrt r}{2+\sqrt r}\le\frac12$.
Therefore every interior index satisfies
$B_{c}(\boldsymbol{x})\le\frac18$, and
\[
\sum_{c\in\mathcal{I}^{\circ}}B_{c}(\boldsymbol{x})^{2}
\le\Bigl(\max_{c\in\mathcal{I}^{\circ}}B_{c}(\boldsymbol{x})\Bigr)
   \sum_{c}B_{c}(\boldsymbol{x})
\le\tfrac18 .
\]
For $\boldsymbol{v}\in\mathcal{C}_{0}$ only interior columns act, so
by the Cauchy--Schwarz inequality applied componentwise,
$\|\boldsymbol{v}\,\boldsymbol{B}(\boldsymbol{x})\|_{2}
\le\|\boldsymbol{v}\|_{\mathcal{C}}
\bigl(\sum_{c\in\mathcal{I}^{\circ}}B_{c}^{2}\bigr)^{1/2}
\le\frac{1}{2\sqrt2}\|\boldsymbol{v}\|_{\mathcal{C}}$;
the $\ell^{1}$ bound follows from
$\|\cdot\|_{1}\le\sqrt3\,\|\cdot\|_{2}$ in $\mathbb{R}^{3}$.

(d) $\|\boldsymbol{v}\,\boldsymbol{w}\|_{1}
\le\sum_{c}\|v_{c}\|_{1}|w_{c}|
\le\bigl(\max_{c}\|v_{c}\|_{1}\bigr)\|\boldsymbol{w}\|_{1}
\le\sqrt3\,\|\boldsymbol{v}\|_{\mathcal{C}}\,\|\boldsymbol{w}\|_{1}$.
\end{proof}

\begin{proposition}
\label{lem:inv}
Let $\boldsymbol{\delta}\in C([0,1],\mathcal{C})$ with
$\boldsymbol{\delta}(t)\in\mathcal{C}_{0}$ for all $t$, and let
$\boldsymbol{x}^{(0)}\in[0,1]^{3}$. Then the initial value problem
$\dot{\boldsymbol{y}}=\boldsymbol{F}_{\boldsymbol{\delta}(t)}(\boldsymbol{y})$,
$\boldsymbol{y}(0)=\boldsymbol{x}^{(0)}$, admits a unique solution on
$[0,1]$, and $\boldsymbol{y}(t)\in[0,1]^{3}$ for all $t\in[0,1]$.
\end{proposition}

\begin{proof}
The right-hand side is polynomial in $\boldsymbol{y}$ with
coefficients continuous in $t$, hence continuous in $(t,\boldsymbol{y})$
and locally Lipschitz in $\boldsymbol{y}$ uniformly for $t\in[0,1]$ on
compact sets; local existence and (forward and backward) uniqueness
follow from the Picard--Lindel\"of theorem. On each face of the cube
the field vanishes identically: e.g.\ at $x=0$ one has
$b_{m,i}(0)=\delta_{i0}$ (Kronecker), so
$\boldsymbol{F}_{\boldsymbol{\delta}(t)}(0,y,z)
=\sum_{j,k}\delta_{0jk}(t)\,b_{n,j}(y)\,b_{o,k}(z)=\boldsymbol{0}$
because $\delta_{0jk}(t)=\boldsymbol{0}$; the remaining five faces are
analogous. Hence every $\boldsymbol{p}\in\partial[0,1]^{3}$ is an
equilibrium for all $t$. If $\boldsymbol{x}^{(0)}\in\partial[0,1]^{3}$
the solution is the constant $\boldsymbol{x}^{(0)}$. If
$\boldsymbol{x}^{(0)}\in(0,1)^{3}$ and the trajectory reached
$\partial[0,1]^{3}$ at some first time $t^{*}$ at a point
$\boldsymbol{p}$, then backward uniqueness through
$(t^{*},\boldsymbol{p})$ would force the trajectory to coincide with
the constant solution $\boldsymbol{p}$, a contradiction. Thus the
trajectory remains in the compact set $[0,1]^{3}$ as long as it
exists, which excludes blow-up and yields existence on all of
$[0,1]$.
\end{proof}

We are now ready to do the proof of the theorem.\\

\begin{proof}
Set
$\boldsymbol{z}:=\boldsymbol{y}_{\boldsymbol{\delta}}
-\boldsymbol{y}_{\boldsymbol{\gamma}}$, so that
$\boldsymbol{z}(0)=\boldsymbol{0}$ and $\boldsymbol{z}\in
C^{1}([0,1],\mathbb{R}^{3})$. Subtracting the two evolution equations
and adding and subtracting
$\boldsymbol{\gamma}(t)\boldsymbol{B}(\boldsymbol{y}_{\boldsymbol{\delta}}(t))$,
\[
\dot{\boldsymbol{z}}(t)
=\bigl(\boldsymbol{\delta}(t)-\boldsymbol{\gamma}(t)\bigr)
 \boldsymbol{B}\bigl(\boldsymbol{y}_{\boldsymbol{\delta}}(t)\bigr)
+\boldsymbol{\gamma}(t)
 \Bigl(\boldsymbol{B}\bigl(\boldsymbol{y}_{\boldsymbol{\delta}}(t)\bigr)
      -\boldsymbol{B}\bigl(\boldsymbol{y}_{\boldsymbol{\gamma}}(t)\bigr)\Bigr).
\]
Both trajectories lie in $[0,1]^{3}$ (Proposition~\ref{lem:inv}), and
$\boldsymbol{\delta}(t)-\boldsymbol{\gamma}(t)\in\mathcal{C}_{0}$, so
Proposition~\ref{lem:basis}(b)--(d) yields the pointwise differential
inequality
\begin{equation}\label{eq:diffineq}
\|\dot{\boldsymbol{z}}(t)\|_{1}
\le
\sqrt{\tfrac{3}{8}}\,
\|\boldsymbol{\delta}(t)-\boldsymbol{\gamma}(t)\|_{\mathcal{C}}
+\Lambda(t)\,\|\boldsymbol{z}(t)\|_{1},
\qquad
\Lambda(t):=2\sqrt{3}\,\bar{m}\,\|\boldsymbol{\gamma}(t)\|_{\mathcal{C}} .
\end{equation}

Since $\boldsymbol{z}(0)=\boldsymbol{0}$, the fundamental theorem of
calculus and the triangle inequality give
$\|\boldsymbol{z}(t)\|_{1}\le\int_{0}^{t}\|\dot{\boldsymbol{z}}(s)\|_{1}\,ds$,
hence by \eqref{eq:diffineq}
\[
\|\boldsymbol{z}(t)\|_{1}
\le a(t)
+\int_{0}^{t}\Lambda(s)\,\|\boldsymbol{z}(s)\|_{1}\,ds,
\qquad
a(t):=\sqrt{\tfrac{3}{8}}\int_{0}^{t}
\|\boldsymbol{\delta}(s)-\boldsymbol{\gamma}(s)\|_{\mathcal{C}}\,ds .
\]
Since $a$ is nondecreasing, Gr\"onwall's inequality in integral form
gives, for all $t\in[0,1]$,
\[
\|\boldsymbol{z}(t)\|_{1}
\le a(t)\,\exp\!\Bigl(\int_{0}^{t}\Lambda(s)\,ds\Bigr)
\le \sqrt{\tfrac{3}{8}}\;e^{A_{1}}
\int_{0}^{1}\|\boldsymbol{\delta}(s)-\boldsymbol{\gamma}(s)\|_{\mathcal{C}}\,ds
\le \sqrt{\tfrac{3}{8}}\;e^{A_{1}}\,
\|\boldsymbol{\delta}-\boldsymbol{\gamma}\|_{\mathbb{V}},
\]
where the last step is the Cauchy--Schwarz inequality in time on the
unit interval. This proves the first claim, and in particular
\begin{equation}\label{eq:sup}
\sup_{t\in[0,1]}\|\boldsymbol{z}(t)\|_{1}
\le \sqrt{\tfrac{3}{8}}\;e^{A_{1}}\,
\|\boldsymbol{\delta}-\boldsymbol{\gamma}\|_{\mathbb{V}} .
\end{equation}

Taking the $L^{2}([0,1])$ norm in $t$ of \eqref{eq:diffineq} and using
the triangle inequality in $L^{2}$,
\[
\Bigl(\int_{0}^{1}\|\dot{\boldsymbol{z}}(t)\|_{1}^{2}\,dt\Bigr)^{1/2}
\le
\sqrt{\tfrac{3}{8}}\,
\|\boldsymbol{\delta}-\boldsymbol{\gamma}\|_{\mathbb{V}}
+\|\Lambda\|_{L^{2}([0,1])}\,
 \sup_{t\in[0,1]}\|\boldsymbol{z}(t)\|_{1}
\le
\sqrt{\tfrac{3}{8}}\,
\bigl(1+A_{2}\,e^{A_{1}}\bigr)\,
\|\boldsymbol{\delta}-\boldsymbol{\gamma}\|_{\mathbb{V}},
\]
where we used $\|\Lambda\|_{L^{2}([0,1])}=A_{2}$ and \eqref{eq:sup}.
Since $\|\cdot\|_{2}\le\|\cdot\|_{1}$ in $\mathbb{R}^{3}$, the same
bounds hold for the Euclidean norms, and therefore
\[
\|\boldsymbol{z}\|_{H^{1}([0,1],\mathbb{R}^{3})}^{2}
=\int_{0}^{1}\|\boldsymbol{z}(t)\|_{2}^{2}\,dt
+\int_{0}^{1}\|\dot{\boldsymbol{z}}(t)\|_{2}^{2}\,dt
\le
\tfrac{3}{8}\Bigl(
e^{2A_{1}}
+\bigl(1+A_{2}\,e^{A_{1}}\bigr)^{2}
\Bigr)\,
\|\boldsymbol{\delta}-\boldsymbol{\gamma}\|_{\mathbb{V}}^{2},
\]
which is the second claim.
\end{proof}
\section{Proof of Theorem 4}
For proving the reverse inequality, we need the following proposition.

\begin{proposition}[The flow is a diffeomorphism of the cube]\label{lem:flow}
Let $\boldsymbol{\alpha}\in C([0,1],\mathcal{C}_{0})$. Then:
\begin{enumerate}
\item[(a)] for every $t\in[0,1]$ the map
      $\boldsymbol{y}_{\boldsymbol{\alpha}}(t,\cdot):[0,1]^{3}\to[0,1]^{3}$
      is a $C^{1}$ diffeomorphism of $[0,1]^{3}$ onto itself;
\item[(b)] the Jacobian
      $\boldsymbol{J}_{\boldsymbol{\alpha}}(t,\boldsymbol{x})
      :=\nabla_{\boldsymbol{x}}\boldsymbol{y}_{\boldsymbol{\alpha}}(t,\boldsymbol{x})$
      satisfies
      \[
      \det\boldsymbol{J}_{\boldsymbol{\alpha}}(t,\boldsymbol{x})
      =\exp\!\left(\int_{0}^{t}
      \operatorname{tr}\Bigl(\boldsymbol{\alpha}(s)\,
      \nabla\boldsymbol{B}\bigl(\boldsymbol{y}_{\boldsymbol{\alpha}}
      (s,\boldsymbol{x})\bigr)\Bigr)ds\right)>0 ;
      \]
\item[(c)] with $\bar{m}=\max\{m,n,o\}$ it holds that
      $\|\nabla\boldsymbol{B}(\boldsymbol{x})\|_{F}\le 2\sqrt{3}\,\bar{m}$
      for all $\boldsymbol{x}\in[0,1]^{3}$, and consequently
      \[
      0<\det\boldsymbol{J}_{\boldsymbol{\alpha}}(t,\boldsymbol{x})
      \le\exp\Bigl(2\sqrt{3}\,\bar{m}\,
      \|\boldsymbol{\alpha}\|_{L^{1}([0,1],\mathcal{C})}\Bigr)
      \qquad\forall\,t\in[0,1],\ \boldsymbol{x}\in[0,1]^{3}.
      \]
\end{enumerate}
\end{proposition}

\begin{proof}
(a) The field $(t,\boldsymbol{x})\mapsto
\boldsymbol{F}_{\boldsymbol{\alpha}(t)}(\boldsymbol{x})$ is continuous,
polynomial in $\boldsymbol{x}$, and $C^{1}$ in $\boldsymbol{x}$ with
derivative continuous in $(t,\boldsymbol{x})$; hence the flow exists,
is unique, and is $C^{1}$ in the initial condition. By
Proposition~\ref{lem:inv}, $[0,1]^{3}$ is forward invariant; running the
same argument for the time-reversed field
$-\boldsymbol{F}_{\boldsymbol{\alpha}(t-s)}$ (whose coefficients again
lie in $\mathcal{C}_{0}$) shows that
$\boldsymbol{y}_{\boldsymbol{\alpha}}(t,\cdot)$ is invertible on
$[0,1]^{3}$ with $C^{1}$ inverse given by the backward flow; in
particular it maps $[0,1]^{3}$ onto $[0,1]^{3}$.

(b) Differentiating the ODE with respect to $\boldsymbol{x}$ gives the
variational equation
$\partial_{t}\boldsymbol{J}_{\boldsymbol{\alpha}}
=\bigl(\boldsymbol{\alpha}(t)\,
\nabla\boldsymbol{B}(\boldsymbol{y}_{\boldsymbol{\alpha}})\bigr)
\boldsymbol{J}_{\boldsymbol{\alpha}}$ with
$\boldsymbol{J}_{\boldsymbol{\alpha}}(0,\boldsymbol{x})=\boldsymbol{I}$,
and the Jacobi--Liouville formula yields the stated exponential
representation, which is strictly positive.

(c) By the derivative decomposition
$b'_{m,i}=m(b_{m-1,i-1}-b_{m-1,i})$, nonnegativity and the partition
of unity on $[0,1]$, each column of $\nabla\boldsymbol{B}$ satisfies
$\|\partial_{d}\boldsymbol{B}\|_{1}\le 2\bar{m}$, $d=1,2,3$
(cf. Proposition~\ref{lem:basis}(b)). Hence
\[
\|\nabla\boldsymbol{B}\|_{F}^{2}
=\sum_{d=1}^{3}\|\partial_{d}\boldsymbol{B}\|_{2}^{2}
\le\sum_{d=1}^{3}\|\partial_{d}\boldsymbol{B}\|_{1}^{2}
\le 12\,\bar{m}^{2}.
\]
By the Cauchy--Schwarz inequality for the Frobenius inner product,
\[
\bigl|\operatorname{tr}\bigl(\boldsymbol{\alpha}(s)\,
\nabla\boldsymbol{B}(\boldsymbol{y}_{\boldsymbol{\alpha}}(s,\boldsymbol{x}))\bigr)\bigr|
\le\|\boldsymbol{\alpha}(s)\|_{\mathcal{C}}\,
\|\nabla\boldsymbol{B}(\boldsymbol{y}_{\boldsymbol{\alpha}}(s,\boldsymbol{x}))\|_{F}
\le 2\sqrt{3}\,\bar{m}\,\|\boldsymbol{\alpha}(s)\|_{\mathcal{C}},
\]
where we used that
$\boldsymbol{y}_{\boldsymbol{\alpha}}(s,\boldsymbol{x})\in[0,1]^{3}$.
Integrating in $s$ and inserting into (b) gives (c).
\end{proof}

\begin{proposition}[Uniform coercivity of the deformed Gram matrix]
\label{prop:gram}
Let $\boldsymbol{\alpha}\in C([0,1],\mathcal{C}_{0})$ and define, for
$t\in[0,1]$,
\[
\boldsymbol{M}_{\boldsymbol{\alpha}}(t)
:=\int_{[0,1]^{3}}
\boldsymbol{B}\bigl(\boldsymbol{y}_{\boldsymbol{\alpha}}(t,\boldsymbol{x})\bigr)\,
\boldsymbol{B}\bigl(\boldsymbol{y}_{\boldsymbol{\alpha}}(t,\boldsymbol{x})\bigr)^{T}
d\boldsymbol{x}
\;\in\;\mathbb{R}^{N\times N},
\qquad N=(m+1)(n+1)(o+1).
\]
Then $\boldsymbol{M}_{\boldsymbol{\alpha}}(t)$ is symmetric positive
definite and
\[
\lambda_{\min}\bigl(\boldsymbol{M}_{\boldsymbol{\alpha}}(t)\bigr)
\;\ge\;\omega_{\boldsymbol{\alpha}}
:=\frac{\exp\Bigl(-2\sqrt{3}\,\bar{m}\,
\|\boldsymbol{\alpha}\|_{L^{1}([0,1],\mathcal{C})}\Bigr)}
{(2m+1)\binom{2m}{m}\,(2n+1)\binom{2n}{n}\,(2o+1)\binom{2o}{o}}
\qquad\forall\,t\in[0,1].
\]
\end{proposition}

\begin{proof}
Fix $t$ and $\boldsymbol{v}\in\mathbb{R}^{N}$. By
Proposition~\ref{lem:flow}(a) the substitution
$\boldsymbol{u}=\boldsymbol{y}_{\boldsymbol{\alpha}}(t,\boldsymbol{x})$
is a $C^{1}$ change of variables of $[0,1]^{3}$ onto itself, with
volume element
$d\boldsymbol{u}
=\det\boldsymbol{J}_{\boldsymbol{\alpha}}(t,\boldsymbol{x})\,
d\boldsymbol{x}$. Hence
\[
\boldsymbol{v}^{T}\boldsymbol{M}_{\boldsymbol{\alpha}}(t)\,\boldsymbol{v}
=\int_{[0,1]^{3}}
\bigl(\boldsymbol{v}\cdot
\boldsymbol{B}(\boldsymbol{y}_{\boldsymbol{\alpha}}(t,\boldsymbol{x}))\bigr)^{2}
d\boldsymbol{x}
=\int_{[0,1]^{3}}
\bigl(\boldsymbol{v}\cdot\boldsymbol{B}(\boldsymbol{u})\bigr)^{2}\,
\det\boldsymbol{J}_{\boldsymbol{\alpha}}
\bigl(t,\boldsymbol{y}_{\boldsymbol{\alpha}}^{-1}(t,\boldsymbol{u})\bigr)^{-1}
d\boldsymbol{u},
\]
and therefore, by Proposition~\ref{lem:flow}(c),
\[
\boldsymbol{v}^{T}\boldsymbol{M}_{\boldsymbol{\alpha}}(t)\,\boldsymbol{v}
\;\ge\;
\exp\Bigl(-2\sqrt{3}\,\bar{m}\,
\|\boldsymbol{\alpha}\|_{L^{1}([0,1],\mathcal{C})}\Bigr)\,
\boldsymbol{v}^{T}\boldsymbol{G}\,\boldsymbol{v},
\qquad
\boldsymbol{G}:=\int_{[0,1]^{3}}
\boldsymbol{B}(\boldsymbol{u})\boldsymbol{B}(\boldsymbol{u})^{T}
d\boldsymbol{u}.
\]
Since the trivariate basis is the tensor product of the univariate
ones and the composite index is ordered in row-major fashion,
$\boldsymbol{G}=\boldsymbol{G}_{m}\otimes\boldsymbol{G}_{n}\otimes
\boldsymbol{G}_{o}$, where
$(\boldsymbol{G}_{m})_{ij}=\int_{0}^{1}b_{m,i}b_{m,j}\,dx$ is the
univariate Bernstein mass matrix. The $\boldsymbol{G}_{m}$ are
symmetric positive definite, and the eigenvalues of a Kronecker
product of symmetric matrices are the products of the eigenvalues of
the factors, so
$\lambda_{\min}(\boldsymbol{G})
=\lambda_{\min}(\boldsymbol{G}_{m})\,
 \lambda_{\min}(\boldsymbol{G}_{n})\,
 \lambda_{\min}(\boldsymbol{G}_{o})$.
By \cite{lyche_p-norm_2000},
$\lambda_{\min}(\boldsymbol{G}_{m})
=\bigl[(2m+1)\binom{2m}{m}\bigr]^{-1}$,
and the claim follows by combining the last three displays.
\end{proof}

We can now do the proof of the Theorem

\begin{proof}
Membership in $\mathbb{W}$ is immediate: both flows and their time
derivatives are continuous on the compact set $[0,1]\times[0,1]^{3}$.
Set $\boldsymbol{z}:=\boldsymbol{y}_{\boldsymbol{\alpha}}
-\boldsymbol{y}_{\boldsymbol{\gamma}}$. Subtracting the two evolution
equations and adding and subtracting
$\boldsymbol{\alpha}(t)\boldsymbol{B}(\boldsymbol{y}_{\boldsymbol{\gamma}})$,
\[
\dot{\boldsymbol{z}}(t,\boldsymbol{x})
=\bigl(\boldsymbol{\alpha}(t)-\boldsymbol{\gamma}(t)\bigr)\,
\boldsymbol{B}\bigl(\boldsymbol{y}_{\boldsymbol{\gamma}}(t,\boldsymbol{x})\bigr)
+\boldsymbol{\alpha}(t)\,
\Bigl(\boldsymbol{B}\bigl(\boldsymbol{y}_{\boldsymbol{\alpha}}(t,\boldsymbol{x})\bigr)
-\boldsymbol{B}\bigl(\boldsymbol{y}_{\boldsymbol{\gamma}}(t,\boldsymbol{x})\bigr)\Bigr).
\]
Multiplying on the right by
$\boldsymbol{B}(\boldsymbol{y}_{\boldsymbol{\gamma}}(t,\boldsymbol{x}))^{T}$
and integrating over $\boldsymbol{x}\in[0,1]^{3}$,
\[
\int_{[0,1]^{3}}\dot{\boldsymbol{z}}\,
\boldsymbol{B}(\boldsymbol{y}_{\boldsymbol{\gamma}})^{T}\,d\boldsymbol{x}
=\bigl(\boldsymbol{\alpha}(t)-\boldsymbol{\gamma}(t)\bigr)\,
\boldsymbol{M}_{\boldsymbol{\gamma}}(t)
+\boldsymbol{\alpha}(t)\,\boldsymbol{H}(t)\]
\[\boldsymbol{H}(t):=\int_{[0,1]^{3}}
\bigl(\boldsymbol{B}(\boldsymbol{y}_{\boldsymbol{\alpha}})
-\boldsymbol{B}(\boldsymbol{y}_{\boldsymbol{\gamma}})\bigr)\,
\boldsymbol{B}(\boldsymbol{y}_{\boldsymbol{\gamma}})^{T}\,d\boldsymbol{x}.
\]
By Proposition~\ref{prop:gram}, $\boldsymbol{M}_{\boldsymbol{\gamma}}(t)$
is invertible with
$\|\boldsymbol{M}_{\boldsymbol{\gamma}}(t)^{-1}\|_{2}
\le\omega_{\boldsymbol{\gamma}}^{-1}$, hence
\[
\boldsymbol{\alpha}(t)-\boldsymbol{\gamma}(t)
=\Bigl(\int_{[0,1]^{3}}\dot{\boldsymbol{z}}\,
\boldsymbol{B}(\boldsymbol{y}_{\boldsymbol{\gamma}})^{T}\,d\boldsymbol{x}
-\boldsymbol{\alpha}(t)\,\boldsymbol{H}(t)\Bigr)\,
\boldsymbol{M}_{\boldsymbol{\gamma}}(t)^{-1}.
\]
We now bound the two factors. First, since
$\boldsymbol{y}_{\boldsymbol{\gamma}}(t,\boldsymbol{x})\in[0,1]^{3}$,
Proposition~\ref{lem:basis}(a) gives
$\|\boldsymbol{B}(\boldsymbol{y}_{\boldsymbol{\gamma}})\|_{2}
\le\|\boldsymbol{B}(\boldsymbol{y}_{\boldsymbol{\gamma}})\|_{1}=1$, so
by the integral Minkowski and Cauchy--Schwarz inequalities on the unit
cube,
\[
\Bigl\|\int_{[0,1]^{3}}\dot{\boldsymbol{z}}\,
\boldsymbol{B}(\boldsymbol{y}_{\boldsymbol{\gamma}})^{T}\,d\boldsymbol{x}
\Bigr\|_{F}
\le\int_{[0,1]^{3}}
\|\dot{\boldsymbol{z}}\|_{2}\,
\|\boldsymbol{B}(\boldsymbol{y}_{\boldsymbol{\gamma}})\|_{2}\,
d\boldsymbol{x}
\le\|\dot{\boldsymbol{z}}(t,\cdot)\|_{L^{2}_{\boldsymbol{x}}} .
\]
Second, by the Lipschitz bound of Proposition~\ref{lem:basis}(b) together
with $\|\cdot\|_{2}\le\|\cdot\|_{1}$ in $\mathbb{R}^{N}$ and
$\|\cdot\|_{1}\le\sqrt3\,\|\cdot\|_{2}$ in $\mathbb{R}^{3}$, and again
Cauchy--Schwarz in $\boldsymbol{x}$,
\[
\|\boldsymbol{H}(t)\|_{F}
\le\int_{[0,1]^{3}}
\bigl\|\boldsymbol{B}(\boldsymbol{y}_{\boldsymbol{\alpha}})
-\boldsymbol{B}(\boldsymbol{y}_{\boldsymbol{\gamma}})\bigr\|_{2}\,
d\boldsymbol{x}
\le 2\sqrt{3}\,\bar{m}\int_{[0,1]^{3}}
\|\boldsymbol{z}\|_{2}\,d\boldsymbol{x}
\le 2\sqrt{3}\,\bar{m}\,
\|\boldsymbol{z}(t,\cdot)\|_{L^{2}_{\boldsymbol{x}}} .
\]
Using $\|\boldsymbol{A}\boldsymbol{M}^{-1}\|_{F}
\le\|\boldsymbol{A}\|_{F}\,\|\boldsymbol{M}^{-1}\|_{2}$ and
$\|\boldsymbol{\alpha}(t)\boldsymbol{H}(t)\|_{F}
\le\|\boldsymbol{\alpha}(t)\|_{\mathcal{C}}\,
\|\boldsymbol{H}(t)\|_{F}$, we obtain, for every $t\in[0,1]$,
\[
\|\boldsymbol{\alpha}(t)-\boldsymbol{\gamma}(t)\|_{\mathcal{C}}
\le\frac{1}{\omega_{\boldsymbol{\gamma}}}
\Bigl(
\|\dot{\boldsymbol{z}}(t,\cdot)\|_{L^{2}_{\boldsymbol{x}}}
+2\sqrt{3}\,\bar{m}\,
\|\boldsymbol{\alpha}\|_{L^{\infty}([0,1],\mathcal{C})}\,
\|\boldsymbol{z}(t,\cdot)\|_{L^{2}_{\boldsymbol{x}}}
\Bigr).
\]
Writing $C:=2\sqrt{3}\,\bar{m}\,
\|\boldsymbol{\alpha}\|_{L^{\infty}([0,1],\mathcal{C})}$ and applying
the Cauchy--Schwarz inequality in $\mathbb{R}^{2}$ to the pair
$(1,C)$,
\[
\|\boldsymbol{\alpha}(t)-\boldsymbol{\gamma}(t)\|_{\mathcal{C}}^{2}
\le\frac{1+C^{2}}{\omega_{\boldsymbol{\gamma}}^{2}}
\Bigl(
\|\dot{\boldsymbol{z}}(t,\cdot)\|_{L^{2}_{\boldsymbol{x}}}^{2}
+\|\boldsymbol{z}(t,\cdot)\|_{L^{2}_{\boldsymbol{x}}}^{2}
\Bigr).
\]
Integrating over $t\in[0,1]$ yields
\[
\|\boldsymbol{\alpha}-\boldsymbol{\gamma}\|_{\mathbb{V}}^{2}
\le\frac{1+C^{2}}{\omega_{\boldsymbol{\gamma}}^{2}}\,
\|\boldsymbol{z}\|_{\mathbb{W}}^{2},
\]
which, upon taking square roots and rearranging, is the asserted
inequality with $1+C^{2}=1+12\,\bar{m}^{2}\,
\|\boldsymbol{\alpha}\|_{L^{\infty}([0,1],\mathcal{C})}^{2}$.
\end{proof}

\section{TARFLOW detailed construction and training}
To rigorously characterise the transformer architecture, we introduce the set $\mathbb{A}^{D}$ defined by
\[
\mathbb{A}^{D}=\left\{\left(n,\boldsymbol{y}\right)\in \left(\{1,\ldots, D\},\bigcup_{i=1}^{D}\mathbb{R}^{i}\right)\,\middle|\, \boldsymbol{y}\in \mathbb{R}^{n} \right\}.
\]
A transformer is then a mapping $f:\mathbb{A}^{D}\rightarrow \mathbb{A}^{D}$ that preserves the first component of elements in $\mathbb{A}^{D}$, thereby keeping the length of the input vector invariant.

An autoregressive flow $\boldsymbol{T}_{j}:\mathbb{R}^{D}\rightarrow\mathbb{R}^{D}$ is defined as
\[
\begin{aligned}
& \tilde{\boldsymbol{z}} \leftarrow \pi\left(\boldsymbol{z}\right), \\
& \boldsymbol{T}_{j}(\boldsymbol{z}) \cdot \boldsymbol{e}^{(i)}\leftarrow \begin{cases}
\tilde{\boldsymbol{z}} \cdot \boldsymbol{e}^{(i)} & i=1, \\[4pt]
\left(\tilde{\boldsymbol{z}} \cdot \boldsymbol{e}^{(i)}-\boldsymbol{H}_{j}\left(\tilde{\boldsymbol{z}} \boldsymbol{I}_{<i}\right)\cdot \boldsymbol{v}^{(i)}\right) 
\exp \left(-\boldsymbol{L}_{j}\left(\tilde{\boldsymbol{z}} \boldsymbol{I}_{<i}\right)\cdot \boldsymbol{v}^{(i)}\right) & i>1,
\end{cases}
\end{aligned}
\]
where $\boldsymbol{L}_{j}:\mathbb{A}^{D}\rightarrow\mathbb{A}^{D}$ and $\boldsymbol{H}_{j}:\mathbb{A}^{D}\rightarrow\mathbb{A}^{D}$ denote transformer models (implemented with causal masking), and $\pi$ denotes the reversal (flip) operator on vectors. The vector $\boldsymbol{v}^{(i)}$ is the $i$-th canonical basis vector of $\mathbb{R}^{i}$, and $\boldsymbol{I}_{<i}$ is an $10\times i$ matrix whose columns are the first $i-1$ canonical basis vectors of $\mathbb{R}^{D}$. Consequently, the product $\boldsymbol{z}\boldsymbol{I}_{<i}$ extracts the first $i-1$ components of $\boldsymbol{z}$.

The inverse transformation is given by
\[
\begin{aligned}
& \tilde{\boldsymbol{z}} \leftarrow \pi^{-1}\left(\boldsymbol{z}\right), \\
& \boldsymbol{T}^{-1}_{j}(\boldsymbol{z})\cdot \boldsymbol{e}^{(i)} \leftarrow \begin{cases}
\tilde{\boldsymbol{z}} \cdot \boldsymbol{e}^{(i)} & i=1, \\[4pt]
\left(\tilde{\boldsymbol{z}} \cdot \boldsymbol{e}^{(i)}\right) 
\exp \left(\boldsymbol{L}_{j}\left(\tilde{\boldsymbol{z}} \boldsymbol{I}_{<i}\right)\cdot \boldsymbol{v}^{(i)}\right)
+ \boldsymbol{H}_{j}\left(\tilde{\boldsymbol{z}} \boldsymbol{I}_{<i}\right)\cdot \boldsymbol{v}^{(i)} & i>1.
\end{cases}
\end{aligned}
\]

 Let $\boldsymbol{T}_{\boldsymbol{\gamma}}$ denote the resulting composition, parametrised by $\boldsymbol{\gamma}$, and let $\tilde{\boldsymbol{T}}_{j}$ denote the composition of the flows $\boldsymbol{T}_{k}$ for $k<j$, with $\tilde{\boldsymbol{T}}_{1}=\operatorname{id}$. For brevity, we omit the explicit dependence on $\boldsymbol{\gamma}$ in the intermediate flows.

Let $\boldsymbol{z}^{(k)}$ denote the encodings of the POD coefficients associated with the instantaneous velocity field $\boldsymbol{\alpha}^{(k)}$. We aim to maximise the log-likelihood of the pushforward of the Gaussian distribution under $\boldsymbol{T}_{\boldsymbol{\gamma}}$. By a change-of-variables formulation, the corresponding optimization problem is \cite{gu_starflow_2025}
\[
\min _{\boldsymbol{\gamma}} \sum_{k=1}^{N}\left( \frac{1}{2}\left\|\boldsymbol{T}_{\boldsymbol{\gamma}}(\boldsymbol{z}^{(k)})\right\|_2^2
+ \sum_{i=1}^{D} \sum_{j=1}^{3} \boldsymbol{L}_{j}\big(\tilde{\boldsymbol{T}}_{j}(\boldsymbol{z}^{(k)})\boldsymbol{I}_{<i}\big)\cdot \boldsymbol{v}^{(i)}\right).
\]

\end{document}